\documentclass[10pt,oneside,a4paper]{article}
\usepackage{amsthm}
\usepackage{titlesec}
	\titleformat{\subsection}[runin]{\normalfont\bfseries}{\thesubsection.}{.5em}{}[. ~]
		\titlespacing{\subsection}{0pt}{1.5ex plus .1ex minus .2ex}{0pt}

\usepackage{graphicx}
\usepackage{subcaption}
\usepackage{wrapfig}
\usepackage{amsfonts,amssymb,amsmath,amsthm}
\usepackage{enumerate}
\usepackage{url}
\usepackage[mathcal,mathscr]{euscript}
\usepackage[dvipsnames]{xcolor}
\usepackage{hyperref}
\hypersetup{colorlinks=true,linkcolor=Magenta,citecolor=PineGreen}

\let\OLDthebibliography\thebibliography
\renewcommand\thebibliography[1]{
  \OLDthebibliography{#1}
  \setlength{\parskip}{0pt}
  \setlength{\itemsep}{0pt plus 0.3ex}
}

\newtheoremstyle{thmstyle}
{\topsep}
{\topsep}
{\itshape}
{0pt}
{\bfseries}
{.}
{5pt plus 1pt minus 1pt}
{#2.\hspace{3pt}#1\thmnote{~\textnormal{(#3)}}}

\newtheoremstyle{defistyle}
{\topsep}
{\topsep}
{}
{0pt}
{\bfseries}
{.}
{5pt plus 1pt minus 1pt}
{#2.\hspace{3pt}#1#3}

\theoremstyle{thmstyle}
\newtheorem{thm}[subsection]{Theorem}
\newtheorem{lemma}[subsection]{Lemma}
\newtheorem{prop}[subsection]{Proposition}
\newtheorem{cor}[subsection]{Corollary}

\theoremstyle{defistyle}
\newtheorem{defi}[subsection]{Definition}
\newtheorem{rmk}[subsection]{Remark}
\newtheorem{notation}[subsection]{Notation}
\newtheorem{example}[subsection]{Example}

\newcommand{\trace}{\mathop{\mathrm{tr}}}
\newcommand{\alttrace}{\mathop{{\mathsf{utr}}}}
\newcommand{\tance}{\mathop{\mathrm{ta}}}
\newcommand{\Arg}{\mathop{\mathrm{Arg}}}

\newcommand{\SU}{\mathop{\mathrm{SU}}}
\newcommand{\PU}{\mathop{\mathrm{PU}}}
\newcommand{\BV}{{\mathrm B}\,V}
\newcommand{\SV}{{\mathrm S}\,V}
\newcommand{\EV}{{\mathrm E}\,V}
\newcommand{\real}{\mathop{\mathrm{Re}}}
\newcommand{\imag}{\mathop{\mathrm{Im}}}
\newcommand{\PCV}{\mathbb{P}_{\mathbb C}V}
\newcommand{\PV}{\mathbb{P}V}

\begin{document}
\title{The length of $\PU(2,1)$ relative to special elliptic isometries with fixed
parameter}
\author{Felipe A. Franco}
\date{}
\maketitle

\begin{abstract}
Generalizing the involution length of the complex hyperbolic plane,
we obtain that the $\alpha$-length of $\PU(2,1)$ is $4$, that is, 
every element of $\PU(2,1)$ can be decomposed as the product of at most $4$
special elliptic isometries with parameter $\alpha$.
We also describe the isometries that can be written as the product of $2$ or $3$
such special elliptic isometries.
\end{abstract}

\section{Introduction}

This work concerns the decomposition of isometries of the complex hyperbolic
plane $\mathbb H_{\mathbb C}^2$ in the product of special elliptic isometries, 
also known as complex reflections. Special elliptic isometries (see 
Subsection~\ref{subsec:isometries}) can be seen as
rotations either around a point or around a complex geodesic. 
They have a {\it center\/} (a nonisotropic point) and a 
{\it parameter\/} $\alpha$ (a unit complex number that is not a cube root of unity);
a special elliptic isometry with center $p$ and parameter $\alpha$ is 
denoted by $R_\alpha^p$. The isometry $R_\alpha^p$ acts on 
$\mathbb H_{\mathbb C}^2$ as a rotation around $p$ (if $p\in\mathbb H_{\mathbb C}^2$) 
by the angle $\Arg\alpha^{3}$ or around a complex geodesic (whose polar point is 
the positive point~$p$) by the angle~$\Arg\alpha^{-3}$. 

Here we approach the problem of finding the smallest number $m$ 
such that every element of $\PU(2,1)$, the group of orientation-preserving
isometries of $\mathbb H_{\mathbb C}^2$, admits a decomposition as the product of 
at most $m$ special elliptic isometries with parameter $\alpha$, for a given $\alpha$.
We call this number the {\it $\alpha$-length\/} of $\PU(2,1)$
(see Definition~\ref{defi:adecomp}). Analogously
we can consider the $\alpha$-length of a given isometry and, in this way, 
the $\alpha$-length of $\PU(2,1)$ is the maximum of the $\alpha$-lengths of all of
its elements.

As orientation-preserving involutions of $\mathbb H_{\mathbb C}^2$ are special 
elliptic isometries with parameter satisfying $\alpha^3=-1$, the $\alpha$-length is closely related to the idea of
{\it involution length\/}:
for a symmetric Riemannian space $X$, the involution length is 
defined similarly to the $\alpha$-length, but considering decompositions into
the products of involutions of $\mathrm{Isom}(X)$.
In~\cite{Will2017} Will and Paupert obtained that the (orientation-preserving)
involution length of $\PU(2,1)$ is $4$ and that the (orientation-preserving)
involution length of $\PU(n,1)$, $n\geq 2$, is at most $8$. 
Allowing orientation-reversing involutions, Falbel and Zocca obtained 
in~\cite{FZ99} that the involution length of $\PU(2,1)$ is $2$, and it turns out 
that this is also true for $\PU(n,1)$, $n\geq 2$ (a proof of this fact 
can be found in~\cite{GT2016}). The involution length of symmetric spaces 
of constant curvature was obtained by Basmajian and Maskit in~\cite{BB2012}.  

In this context, for fixed parameters $\alpha_1$ and $\alpha_2$,
we determine every isometry of $\mathbb H_{\mathbb C}^2$ that admits a decomposition 
as the product of two special elliptic isometries, one with parameter $\alpha_1$
and the other with parameter $\alpha_2$ (Propositions~\ref{prop:a1a2loxpardecomp}, and~\ref{prop:a1a2eldecomp}). (In $\SU(2,1)$, 
such a decomposition is of the form $F=\delta R_{\alpha_2}^{p_2}R_{\alpha_1}^{p_1}$,
where $\delta$ is a cube root of unity.) The isometries admitting
such a decomposition are related to lines tangent to Goldman's deltoid 
(see Subsection~\ref{subsec:tangentlines}). Moreover,
deciding if an elliptic isometry admits such a decomposition is more
involved and gives rise to the study of the interaction between
these tangent lines and the {\it unfolded
trace\/}, introduced in Subsection~\ref{subsec:alttrace} --- a trace-like function 
that, unlike the usual trace, can distinguish the classes of 
regular elliptic isometries. This interaction is also central to
obtain our main result, Theorem~\ref{thm:alphalength}. The unit complex
numbers in the following theorem are considered up to a cube root of unity
(see Corollary~\ref{cor:deltaalpha}).

\begin{thm}
\label{thm:alphalength}
The\/ $\alpha$-length of\/ $\PU(2,1)$ is\/ $4$, for any
parameter\/ $\alpha$. Moreover, writing\/
$\alpha=e^{ai}$, if\/
$0<a\leq \frac{4\pi}{27}$ or\/ $\frac{14\pi}{27}\leq a<\frac{2\pi}3$,
then every isometry, except possibly\/ $2$-step unipotent ones\/ 
{\rm(}see\/ {\rm Remark~\ref{rmk:2stepdecomp})}, admits 
a decomposition as the product of three special elliptic isometries with parameter\/ 
$\alpha$. When\/ $\frac{4\pi}{27}<a<\frac{14\pi}{27}$, there exist regular and
special elliptic isometries that do not admit such a decomposition. 
Finally, 
if\/ $\alpha^3=\pm 1$, i.e.,\/ $a=\frac{2\pi}9,\frac{\pi}3,\frac{4\pi}9$, then every\/ 
$2$-step unipotent isometry is the product of
three special elliptic isometries with parameter\/~$\alpha$.
\end{thm}

The general strategy to prove this theorem follows~\cite{spell} and~\cite{Will2017}.
In~\cite{spell} it was described every (generic) relation of length at most~$4$
between special elliptic isometries and the space 
$S_{\pmb\alpha,\pmb\sigma,\tau}$ of (strongly regular) triples 
$p_1,p_2,p_3$ of nonisotropic points that
satisfy $\trace R_{\alpha_3}^{p_3}R_{\alpha_2}^{p_2}R_{\alpha_1}^{p_1}=\tau$ for a fixed
triple of parameters $\pmb\alpha=(\alpha_1,\alpha_2,\alpha_3)$, a fixed triple of
signs $\pmb\sigma=(\sigma_1,\sigma_2,\sigma_3)$ (the signatures $\sigma_i$ of the points
$p_i$), and fixed~$\tau\in\mathbb C$.
Here we prove that, given a triple of parameters $\pmb\alpha$ and $\tau\in\mathbb C$,
there exists a triple of signs $\pmb\sigma$ such that the space 
$S_{\pmb\alpha,\pmb\sigma,\tau}$ is nonempty (proof of Proposition~\ref{prop:parlox3}). 
From the nature of the space of $\PU(2,1)$-conjugacy classes 
(see~Subsection~\ref{subsec:puclasses}), this is sufficient to prove that the
$\alpha$-length of regular parabolic (ellipto-parabolic or $3$-step unipotent, see 
Definition~\ref{defi:regular}) and loxodromic isometries is at most $3$,
for any parameter $\alpha$ (Proposition~\ref{prop:parlox3}), but this is not
sufficient for the general case (particularly, for elliptic isometries).

To deal with the decomposition of the remaining elliptic isometries, 
following~\cite{Will2017}, we use the {\it product map\/} 
$\widetilde\mu:C_1\times C_2\to\mathcal G$, where~$\mathcal G$ denotes the space 
of all $\PU(2,1)$-conjugacy classes, $C_i$ denotes semisimple conjugacy classes, 
$i=1,2$, and $\widetilde\mu(A,B):=[AB]$, where $[F]$ stands for the conjugacy class 
of~$F$. (As will be clear in Subsection~\ref{subsec:productmap}, we actually consider
the projection~$\overline\mu$ of the product map
over the maximal Hausdorff quotient~$c(\mathcal G)$ of $\mathcal G$.)

The idea is the following: suppose that $C_1$ is the class of a special elliptic isometry 
with parameter $\alpha$ and $C_2$ is the class of a semisimple isometry that admits a 
decomposition as the product of $2$ special elliptic isometries with parameter~$\alpha$.
Then any isometry $F$ with $[F]\in\widetilde\mu(C_1\times C_2)$ admits a decomposition
as the product of $3$ special elliptic isometries.
In this way, to obtain all elliptic isometries admitting such length $3$ decomposition, 
it remains to describe the union of all possible images 
of the product map, varying the semisimple classes $C_1$ and $C_2$.
Properties of the product map $\overline\mu$ and its image (and mainly 
its intersection with $\rho(\mathcal E)$, where $\mathcal E$ is the space of 
elliptic conjugacy classes) where described in~\cite{FW2009,Paupert2007} and we 
list some of this properties in Subsection~\ref{subsec:productmap}.
If at most one of the classes $C_1$ and $C_2$ is the class of a special elliptic 
isometry, then the image $\overline\mu(C_1\times C_2)$ is the union of
closed {\it chambers\/} bounded by {\it reducible walls\/}; the intersection
of the reducible walls with $\rho(\mathcal E)$
is composed by finitely many line segments of slopes $-1,\frac{1}2,2$.
The interaction between tangent lines to Goldman's deltoid and the unfolded trace
help us describe the union of the images mentioned above and cast light over the 
nature of the segments of slopes $-1,\frac{1}2,2$ in $\rho(\mathcal E)$:
the image under the unfolded trace of segments of slopes $-1,\frac{1}2,2$
in $\mathcal E$ are subsegments of lines tangent to the deltoid. 

The last ingredient to prove the main theorem are {\it bending relations\/} 
(see Subsection~\ref{subsec:fullempty} and \cite[Section~4]{spell}). Such relations can 
be used to deform a given decomposition of the form
$F=\delta R_{\alpha_n}^{p_n}\ldots R_{\alpha_1}^{p_1}$, where $\delta$ is a cube root 
of unity, and they appear as natural coordinates of the 
space~$S_{\pmb\alpha,\pmb\sigma,\tau}$ (see~\cite[Theorem~5.4]{spell}).
In the context of Theorem~\ref{thm:alphalength}, bending the decomposition 
$F=\delta R_{\alpha}^{p_3}R_{\alpha}^{p_2}R_{\alpha}^{p_1}$, we can decide 
whether~$[F]$ is an interior point in the union of all images of~$\widetilde\mu$ or not.
This allows us to avoid having to directly obtain such a union, as is done in 
\cite{Will2017},
and leads to the description in Proposition~\ref{prop:a3walls} which, together with 
the results in Subsection~\ref{subsec:fullchambers}, proves our main theorem. 

\subsection*{Acknowledgements} I would like to thank Carlos Grossi and Samanta
Silva for the discussions relating to this work. 
Also, I would like to thank the anonymous referee for the many useful comments
and suggestions.

\section{Complex hyperbolic geometry}
\label{sec:hypgeo}

In this section, following \cite{SashaGrossi2011,SGG2011,Goldman1999}, we 
briefly describe the complex hyperbolic plane and its isometries.

Let $V$ be a $3$-dimensional $\mathbb C$-linear space equipped with
a Hermitian form of signature $++-$. We consider the projectivization
$\PV:=\PCV$ divided into negative, positive and isotropic points:
$$\BV := \left\{p\in{\mathbb P}V \,|\, \left<p,p\right> <0\right\},\ \ 
\SV := \left\{p\in{\mathbb P}V \,|\, \left<p,p\right> =0\right\},\ 
\EV:=\{p\in\mathbb PV\,|\, \left<p,p\right>>0\}.$$
Here and throughout this paper we denote a 
point in $\PV$ and a representative of it in $V$ by the same letter, but no
confusion should arise.
We denote the {\it signature}\/ of a point $p\in\PV$ by $\sigma p$, i.e.,
$\sigma p$ is respectively $-1,0,1$ if $p$ is negative, isotropic, or positive.

If $p\in\PV$ is a nonisotropic point, we have the identification
$\mathrm T_p\mathbb PV\simeq{\mathrm{Lin}}_{\mathbb C}(\mathbb Cp,p^\bot)$,
where $p^\perp$ is the linear subspace of $V$
orthogonal to $p$. Through this identification we can define
a Hermitian metric in $\BV$ and $\EV$ by
$$\langle t_1,t_2\rangle:=
-\frac{\langle t_1(p),t_2(p)\rangle}{\langle p,p\rangle},$$
where $t_1,t_2\in\mathrm{Lin}_{\mathbb C}(\mathbb Cp,p^\perp)$ are tangent
vectors to $\PV$ at $p$. This metric is positive
definite in $\BV$ and has signature $+-$ in $\EV$. Thus, the real part of
this metric defines a Riemannian metric in~$\BV$ and a pseudo-Riemannian
metric in $\EV$. The $4$-ball $\BV$ equipped with such Riemannian metric is the
{\it complex hyperbolic plane\/}~$\mathbb H_\mathbb C^2$. Its ideal boundary, also 
called {\it absolute\/}, is the $3$-sphere~$\SV$.

\smallskip

The projectivization $\mathbb PW$ of a $2$-dimensional complex subspace
$W\leq V$ is called a {\it complex line\/}. Given a complex line $L$,
the point $c\in\mathbb PV$ such that $L=\mathbb Pc^\perp$ is the {\it polar
point\/} of $L$. We say that a complex line $L=\mathbb Pc^\perp$ is
{\it hyperbolic\/}, {\it spherical\/}, {\it Euclidean\/} if $c\in\EV$, $c\in\BV$,
$c\in\SV$, respectively. A {\it complex geodesic\/} is a set of the form
$L\cap\BV$, where $L$ is a hyperbolic complex line. For distinct points
$p_1,p_2\in\PV$, we denote by $\mathrm L(p_1,p_2)$ the complex line
$\mathbb P(\mathbb Cp_1+\mathbb Cp_2)$.

The {\it tance\/} between two nonisotropic points $p_1,p_2\in\mathbb PV\setminus\SV$
is given by
$$\tance(p_1,p_2):=\frac{\langle p_1,p_2\rangle\langle p_2,p_1\rangle}
{\langle p_1,p_1\rangle\langle p_2,p_2\rangle}.$$
By Sylvester's criterion, the line $\mathrm L(p_1,p_2)$, with $p_1,p_2$
distinct nonisotropic points in $\PV$, is hyperbolic
iff $\tance(p_1,p_2)>1$ or $\tance(p_1,p_2)<0$; spherical iff $0<\tance(p_1,p_2)<1$;
and Euclidean iff $\tance(p_1,p_2)=1$.

\subsection{Isometries of the complex hyperbolic plane}
\label{subsec:isometries}
Consider the special unitary group $\SU(2,1)$ given by the elements in 
$\mathrm{GL}\,V$ that preserve the Hermitian form of $V$ and have determinant $1$. 
The group of orientation-preserving isometries of $\mathbb H_{\mathbb C}^2$
is the projectivization $\PU(2,1)$ of $\SU(2,1)$, i.e.,
$\PU(2,1)=\SU(2,1)/\{1,\omega,\omega^2\}$, where $\omega:=e^{2\pi i/3}$.
We also refer to elements of $\SU(2,1)$ as isometries.

We say that a nonidentical isometry in $\PU(2,1)$ is {\it elliptic}\/ if it
fixes a point in $\BV$, {\it parabolic}\/ if it fixes
exactly one point in $\SV$, and {\it loxodromic}\/ if it fixes exactly two points in
$\SV$. Elliptic and parabolic isometries are further
divided into subtypes as follows.

\smallskip

Let $I\in\SU(2,1)$ be an elliptic isometry and let $c\in\BV$ be an
$I$-fixed point. Thus, the spherical complex line $\mathbb P c^\perp$ is
$I$-stable and $I$ fixes another point $p$ in this line. Clearly, 
$I$ must also fix the point $\tilde{p}\in\mathbb Pc^\perp$ orthogonal to $p$.
We obtain an orthogonal basis $\{c,p,\tilde{p}\}$ for $V$ given by eigenvectors
of $I$. If $\mu_1,\mu_2,\mu_3\in\mathbb C$ are the eigenvalues of 
$c,p,\tilde{p}$, respectively, we have $\mu_1\mu_2\mu_3=1$ and 
$|\mu_i|=1$. We say that the elliptic isometry $I$ is {\it regular}\/ 
if the eigenvalues $\mu_i$ are pairwise distinct; 
otherwise, we say that it is {\it special}.

Throughout the paper we will denote by $\mathbb S^1$ the set
of unit complex numbers and by~$\Omega:=\{1,\omega,\omega^2\}$,
where $\omega:=e^{2\pi i/3}$, the set of cubic roots of unity. 
Every special elliptic isometry can be written in the form (see~\cite{Mostow80})
\begin{equation}
\label{eq:specialelliptic}
R_\alpha^p:x\mapsto (\alpha^{-2}-\alpha){\frac{\left<x,\,p\right>}{\left<p,\,p\right>}}p+
\alpha x
\end{equation}
for some $p\in\PV\setminus\SV$ and $\alpha\in\mathbb S^1\setminus\Omega$. We say that
$p$ is the {\it center\/} and that $\alpha$ is the {\it parameter\/} of~$R_\alpha^p$.

\smallskip

Parabolic isometries are divided into three subtypes. We say that a parabolic isometry 
is unipotent if it lifts to a unipotent element of $\SU(2,1)$. Unipotent isometries 
can be either {\it $2$-step\/} or {\it $3$-step unipotent\/}; 
$2$-step unipotent isometries fix an isotropic point and pointwise
fix its Euclidean polar complex line, and the $3$-step ones fix an isotropic 
point and no other point (so, they do not preserve any hyperbolic complex line). 
Parabolic isometries that are not unipotent are called 
{\it ellipto-parabolic\/}; they fix an isotropic point and also fix a positive point in
the Euclidean stable line.

\smallskip

Consider the function $f:\mathbb C\to\mathbb R$ given by
$$f(z):=|z|^4-8\real(z^3)+18|z|^2-27,$$
and denote $\Delta:=\{z\in\mathbb C\mid f(z)\leq 0\}$,
$\Delta^{\circ}:=\{z\in\mathbb C\mid f(z)<0\}$, and
$\partial\Delta:=\{z\in\mathbb C\mid f(z)=0\}$.
Given a nonidentical isometry $I\in\SU(2,1)$, we have:

\smallskip

$\bullet$ $I$ is regular elliptic iff $\trace I\in\Delta^\circ$;

$\bullet$ $I$ is loxodromic iff $\trace I\in\mathbb C\setminus\Delta$;

$\bullet$ if $I$ is elliptic, then it is special elliptic iff $\trace I\in\partial\Delta$;

$\bullet$ $I$ is parabolic iff $I$ is not elliptic and $\trace I\in\partial\Delta$.
Moreover, if $\trace I\not\in\{3,3\omega,3\omega^2\}$, then $I$ is ellipto-parabolic;

$\bullet$ $I$ is unipotent iff $I\in\{3,3\omega,3\omega^2\}$.

\begin{defi}
\label{defi:regular}
An isometry is {\it regular\/} if its eigenspaces have dimension $1$, i.e., it does 
not pointwise fix a complex line. (This definition coincides with the one
in~\cite{steinberg1974}.)
\end{defi}

In other words, a nonidentical isometry in $\PU(2,1)$ is regular if it is 
neither special elliptic nor $2$-step unipotent.

\section{Conjugacy classes and the product map}
\label{sec:conjclasses}

In this section we describe the space of $\PU(2,1)$-conjugacy classes. 
Differently from the case of the Poincar\'e disk, the trace of an isometry
does not determine its conjugacy class (but, as we will see here, it `almost' does).

\smallskip

The trace determines the $\SU(2,1)$-conjugacy class of loxodromic isometries, i.e., 
two loxodromic isometries in $\SU(2,1)$ with the same trace
$\tau\in\mathbb C\setminus\Delta$ are $\SU(2,1)$-conjugated.
Now, this is not true for values of trace in $\Delta$. 

Given two elliptic isometries $F_1,F_2\in\SU(2,1)$ with same trace
or, equivalently, same eigenvalues (see \cite[Proof of Lemma~6.2.5]{Goldman1999}),
$F_1$ and $F_2$ are 
$\SU(2,1)$-conjugated iff their negative fixed points have
the same eigenvalue. Therefore, elliptic classes are distinguished by the
{\it types\/} of its eigenvalues: we say that an eigenvalue $\alpha\in\mathbb S^1$ 
of an isometry $F\in\SU(2,1)$ is of {\it negative type\/} if there exists a
negative eigenvector associated with $\alpha$, i.e., if there exists $v\in V$ 
with $Fv=\alpha v$ and $\langle v,v\rangle<0$. As discussed above, two elliptic isometries
with same trace and same negative type eigenvalue are conjugated.
Since regular elliptic isometries have three distinct eigenvalues (three
possible values for the negative type eigenvalue), for each
$\tau\in\Delta^\circ$, there exists three distinct $\SU(2,1)$-conjugacy classes
of trace $\tau$. 

Now, for $\tau\in\partial\Delta\setminus\{3,3\omega,3\omega^2\}$
(remember that $\omega:=e^{2\pi i/3}$), there exists three 
distinct $\SU(2,1)$-conjugacy classes of isometries with trace $\tau$. 
Two distinct classes of special elliptic isometries (distinguished as above by their
negative type eigenvalue or, equivalently, by the signature of their centers), 
and one class of ellipto-parabolic isometries.

Finally, for $\tau\in\{3,3\omega,3\omega^2\}$, we have three distinct 
nonidentical $\SU(2,1)$-conjugacy classes of isometries with trace~$\tau$: two classes 
of $2$-step unipotent isometries, and one of $3$-step unipotent isometries.

\smallskip

Now we focus our attention on $\PU(2,1)$-conjugacy classes. Note that,
given $\tau\in\mathbb C\setminus\Delta$,
the three loxodromic $\SU(2,1)$-conjugacy classes of traces 
$\tau$, $\omega\tau$, $\omega^2\tau$ determine the same loxodromic 
$\PU(2,1)$-conjugacy class, as they differ
(up to conjugacy) by an element of $\Omega$. The three nonidentical distinct 
$\SU(2,1)$-conjugacy classes of trace $\tau\in\Delta$, determine three distinct 
$\PU(2,1)$-classes when $\tau\neq 0$, and these classes coincide with the ones 
determined by the traces $\omega\tau$ and $\omega^2\tau$; however, for 
$\tau=0=1+\omega+\omega^2$, they
determine the same $\PU(2,1)$-conjugacy class.

\subsection{The space of {\rm PU(2,1)}-conjugacy classes}
\label{subsec:puclasses}
Our approach here follows 
\cite{FW2009} and \cite{Will2017}.
Let $\mathcal G$ be the space of all $\PU(2,1)$-conjugacy classes, i.e.,
the quotient of $\PU(2,1)$ by the action of $\PU(2,1)$ on itself by conjugacy,  
equipped with the quotient topology.

Let $\mathcal G^{\mathrm{reg}}$ be the space of classes of regular isometries
(see Definition~\ref{defi:regular}), let $\mathcal E\subset\mathcal G$ be the space 
of elliptic conjugacy classes (including the 
identical one), let $\mathcal B$ be the space of boundary classes (classes of parabolic
or special elliptic isometries), and let 
$\mathcal L\subset\mathcal G$ be the space of loxodromic conjugacy classes. 
We denote by $\mathcal E^{\mathrm{reg}}\subset\mathcal G^{\mathrm{reg}}$
the space of regular elliptic conjugacy classes. Note that 
$\mathcal G^{\mathrm{reg}}=\mathcal E^{\mathrm{reg}}\cup\mathcal L$. 

\begin{figure}[!ht]
\centering
\includegraphics[scale=.8]{pics/ellipticclass1.mps}
\hspace{1cm}
\includegraphics[scale=.8]{pics/ellipticclass2.mps}
\end{figure}

An elliptic isometry $F\in\PU(2,1)$ stabilizes two orthogonal complex geodesics in
$\mathbb H_\mathbb C^2$, acting as a
rotation by $\theta_1$ on one of them and as a rotation by $\theta_2$ on the other.
We call the (nonoriented) pair $\{\theta_1,\theta_2\}$ the {\it angle pair\/}
of $F$. Two elliptic isometries with same angle pair are $\PU(2,1)$-conjugated.
Hence, $\mathcal E$ is the space of nonoriented angle pairs, and it can be seen as
the triangular region $\mathsf T:=\{(\theta_1,\theta_2)\in\mathbb R^2\mid 0
\leq \theta_2\leq\theta_1\leq 2\pi\}$
quotiented by the identification $(\theta,0)\simeq(2\pi,\theta)$.
Clearly, every point in $\mathbb R^2$ has a representative in
$\mathcal E$. In what follows, if we write an element of 
$\mathcal E$ as $(\theta_1,\theta_2)$,  with parenthesis, we are assuming 
that $0\leq\theta_2\leq\theta_1\leq 2\pi$.
Note that $\mathcal E^{\mathrm{reg}}$ is homeomorphic to the interior of $\mathsf T$.
Also, points of the form $(\theta,\theta)$, for $0<\theta<2\pi$, correspond
to classes of special elliptic isometries with negative center, while
those of the form $(0,\theta)\simeq(2\pi,\theta)$ correspond to classes
of special elliptic isometries with positive center.

\smallskip

As two loxodromic isometries admitting lifts with
same trace are $\PU(2,1)$-conjugated, $\mathcal L$ is homeomorphic
to $\mathbb C\setminus\Delta$ quotiented by the action of 
the subgroup $\Omega\subset\mathbb C$, i.e., it is homeomorphic the cylinder
$\mathbb S^1\times\mathbb R_{>0}$.

\smallskip

Note that the space $\mathcal G$ is not Hausdorff as, for
example, any neighborhood of a unipotent class always intersect a neighborhood of 
the identical class. So, in order to use standard topological arguments, it is 
useful to consider the maximal Hausdorff quotient $c(\mathcal G)$ of $\mathcal G$ 
and its natural projection $\rho:\mathcal G\to c(\mathcal G)$.


Since the space $\mathcal E\cup\mathcal B$ is compact in $\mathcal G$, if follows
that $\rho(\mathcal E)=\rho(\mathcal E\cup\mathcal B)$ is compact in $c(\mathcal G)$.
(Note that $\mathcal E$ is not compact in $\mathcal G$, since there are sequences
of regular elliptic isometries converging to parabolic ones.)
Moreover, $\rho(\mathcal E)$ is homeomorphic to the quotient $\mathsf T/\simeq$.
Such homeomorphism is given by the description of $\mathcal E$ as the space of
angle pairs, together with $\rho$.

In order to understand how $\rho(\mathcal B)$ is identified with the sides of 
$\mathsf T/\simeq$, we observe that 
ellipto-parabolic isometries also have their angle pairs: if $F$ is a parabolic isometry
with repeated eigenvalue $e^{\theta i}\in\mathbb S^1$, then $e^{-2\theta i}$ is 
also an eigenvalue of $F$ and it is associated to the positive fixed point of $F$.
Such an isometry stabilizes a complex geodesic,
where it acts as a parabolic isometry in the sense of the geometry
of the Poincar\'e disk, and
rotates points around this line by the angle~$-3\theta$. Therefore,
the angle pair of $F$ is $\{-3\theta,0\}$. Two parabolic isometries
with the same angle pair are $\PU(2,1)$-conjugated.

So, we identify $\rho(\mathcal E)$ with the quotient
$\mathsf T/\simeq$ by seeing points in $\rho(\mathcal E)$ as angles pairs
corresponding either to regular elliptic classes or to
special elliptic and parabolic classes. Furthermore, we identify the set 
$\mathcal E^{\mathrm{reg}}$
with $\rho(\mathcal E^{\mathrm{reg}})$ and frequently consider 
$\mathcal E^{\mathrm{reg}}=\mathsf T^\circ\subset
\rho(\mathcal E)$.
Using these identifications, 
we will consider the space $\rho(\mathcal E)\subset c(\mathcal G)$ as having two sides
and one vertex  determined by the sides and vertices of~$\mathsf T$; the sides and 
the vertex of $\rho(\mathcal E)$ constitute the set $\rho(\mathcal B)$. The points 
$(0,0)$, $(2\pi,0)$, $(2\pi,2\pi)$ of $\mathsf T$ are identified in the vertex of 
$\rho(\mathcal E)$. The {\it diagonal\/}
(resp. the {\it nondiagonal\/}) {\it side\/} of $\rho(\mathcal E)$ is given by 
angle pairs of the form $(\theta,\theta)$ (resp. $(\theta,0)\simeq (2\pi,\theta)$),
with $0<\theta<2\pi$.

Hence, the fiber of $\rho$ over an angle pair in the nondiagonal side of
$\rho(\mathcal E)$ contains, besides the mentioned special
elliptic class with positive center, the class of ellipto-parabolic isometries
with this angle pair. Furthermore, the fiber of $\rho$ over a point in the 
diagonal side of $\rho(\mathcal E)$ corresponds only to the class of a 
special elliptic isometry with negative center. Finally, the fiber of $\rho$ 
over the vertex of  $\rho(\mathcal E)$ has $4$ elements: the identical class, 
the two classes of $2$-step unipotent isometries, and the class of $3$-step 
unipotent isometries. 

\subsection{Decomposing isometries}
\label{subsec:specdecomp}

Here we discuss what it means to decompose, in $\PU(2,1)$ and $\SU(2,1)$,
an isometry as the product of special elliptic isometries and introduce some notation.

\begin{defi}
\label{defi:adecomp}
Given parameters $\alpha_1,\ldots,\alpha_n\in\mathbb S^1\setminus\Omega$, we 
say that an isometry in $\PU(2,1)$ admits an 
{\it $(\alpha_1,\ldots,\alpha_n)$-decomposition\/} if it has a lift 
$F\in\SU(2,1)$ such that $F=R_{\alpha_n}^{p_n}\ldots R_{\alpha_1}^{p_1}$, for points
$p_1,\ldots,p_n\in\PV\setminus\SV$.
We say that an isometry admits an {\it $\alpha^{(n)}$-decomposition\/}, for 
$\alpha\in\mathbb S^1\setminus\Omega$, if it admits an
$(\alpha,\ldots,\alpha)$-decomposition (with $n$-terms).

Given an isometry $F\in\PU(2,1)$, the smallest number $n\in\mathbb N$ such that 
$F$ admits an $\alpha^{(n)}$-decomposition is the {\it $\alpha$-length\/} of $F$.
The maximum of all $\alpha$-lengths over isometries in $\PU(2,1)$ is the
{\it $\alpha$-length\/} of $\PU(2,1)$.
\end{defi}

\begin{prop}
\label{prop:decomp}
Given parameters\/ $\alpha_i\in\mathbb S^1\setminus\Omega$, the following statements
hold:

\begin{enumerate}[{\rm(}i\/{\rm)}]
\item If an isometry in\/ $\PU(2,1)$ admits an\/
$(\alpha_1,\ldots,\alpha_n)$-decomposition, then every isometry
in its\/ $\PU(2,1)$-conjugacy class also does.
\item If an isometry in\/ $\PU(2,1)$ admits an\/ 
$(\alpha_1,\ldots,\alpha_n)$-decomposition, then it also admits a\/
$(\delta_1\alpha_1,\ldots,\delta_n\alpha_n)$-decomposition,
for any\/ $\delta_1,\ldots,\delta_n\in\Omega$.
\item If an isometry in\/ $\PU(2,1)$ admits an\/ 
$(\alpha_1,\ldots,\alpha_n)$-decomposition, then it also admits a\/
$(\beta_1,\ldots,\beta_n)$-decomposition, where\/ $(\beta_1,\ldots,\beta_n)$ is a 
cyclic permutation of\/ $(\alpha_1,\ldots,\alpha_n)$.
\end{enumerate}
\end{prop}

\begin{proof}
Let $F\in\SU(2,1)$ be an isometry that can be written as
$F=R_{\alpha_n}^{p_n}\ldots R_{\alpha_1}^{p_1}$.

For any $I\in\SU(2,1)$, we have $IFI^{-1}=
R_{\alpha_n}^{Ip_n}\ldots R_{\alpha_1}^{Ip_1}$ since
$IR_\alpha^p I^{-1}=R_\alpha^{Ip}$. This proves~({\it i\/}).

Since $R_{\delta\beta}^q=\delta R_\beta^q$ for any
special elliptic isometry $R_\beta^q$ and any $\delta\in\Omega$,
given $\delta_1,\ldots,\delta_n\in\Omega$ we have
$F=\delta R_{\delta_n\alpha_n}^{p_n}\ldots R_{\delta_1\alpha_1}^{p_1}$,
where $\delta:=\Pi\overline\delta_i\in\Omega$, 
which proves ({\it ii}\/).
Finally, note that $R_{\alpha_1}^{p_1}F R_{\overline\alpha_1}^{p_1}=
R_{\alpha_1}^{p_1}R_{\alpha_n}^{p_n}\ldots R_{\alpha_2}^{p_2}$.
Using ({\it i}\/), we obtain ({\it iii}\/). 
\end{proof}

Using item~({\it i}\/) of Proposition~\ref{prop:decomp}, we can say that a 
conjugacy class admits an $(\alpha_1,\ldots,\alpha_n)$-decomposition,
for given parameters $\alpha_1,\ldots,\alpha_n\in\mathbb S^1\setminus\Omega$;
this means that one (and therefore every) isometry in such class admits an
$(\alpha_1,\ldots,\alpha_n)$-decomposition. In the same way, we can consider the $\alpha$-length
of conjugacy classes in $\mathcal G$.

\begin{notation}
\label{notation:egclasses}
We denote by $\mathsf G_{\alpha_1,\ldots,\alpha_n}\subset\mathcal G$ 
(see Subsection~\ref{subsec:puclasses}) the set of semisimple classes admitting an 
$(\alpha_1,\ldots,\alpha_n)$-decomposition. 
We also denote by $\mathsf E_{\alpha_1,\ldots,\alpha_n}\subset c(\mathcal G)$ the projection 
under~$\rho$ of the set of classes in 
$\mathcal E\cup\mathcal B$ admitting an 
$(\alpha_1,\ldots,\alpha_n)$-decomposition (note that
$\mathsf E_{\alpha_1,\ldots,\alpha_n}\subset\rho(\mathcal E)$).
Following Definition~\ref{defi:adecomp}, we use 
$\mathsf G_{\alpha^{(n)}}$ and $\mathsf E_{\alpha^{(n)}}$ when 
$\alpha_1=\cdots=\alpha_n:=\alpha$. Moreover,
we denote $\mathsf S_{\alpha}$ the set composed by the two classes of special elliptic
isometries of parameter~$\alpha$.
\end{notation}

The next result is a direct consequence of item~({\it ii}\/) of 
Proposition~\ref{prop:decomp} and implies that we can focus our attention
in finding the $\alpha$-length of $\PU(2,1)$ for parameters~$\alpha=e^{ai}$ with~$0<a<2\pi/3$.

\begin{cor}
\label{cor:deltaalpha}
The\/ $\alpha$-length of\/ an isometry in\/ $\PU(2,1)$ is equal to its\/ $\delta\alpha$-length
of\/ $\PU(2,1)$, for any\/ $\delta\in\Omega$.
\end{cor}

Therefore, for any cube root of unity $\delta\in\Omega$, the $\alpha$-length
and the $\delta\alpha$-length of $\PU(2,1)$ coincide.

\section{The unfolded trace and lines tangent to the deltoid}
\label{sec:altracetangent}

In this section we introduce the tools that will compose, together
with the product map (see Subsection~\ref{subsec:productmap}), the technique used 
to obtain the $\alpha$-length of $\PU(2,1)$.
The main idea is the interaction between lines tangent
to the deltoid $\partial\Delta$ (and their relation to the eigenvalues of 
isometries in $\SU(2,1)$)
and the {\it unfolded trace} --- a trace-like function that can distinguish
regular elliptic $\PU(2,1)$-conjugacy classes. 

\subsection{Unfolding the trace of elliptic isometries}
\label{subsec:alttrace}
As seen in Section~\ref{sec:conjclasses}, there are three distinct 
$\PU(2,1)$-conjugacy classes corresponding to the three traces
$\tau,\omega\tau,\omega^2\tau$, where~$\omega:=e^{2\pi i/3}$ 
and~$\tau\in\Delta\setminus\{0\}$. 
So, we have enough space
to `unfold' the trace of $\SU(2,1)$ into a function that
(while coinciding with the trace in some sense) distinguishes 
regular elliptic conjugacy classes. Such a function should, for 
any~$\tau\in\Delta^\circ$, continuously send each of the three classes determined by
$\tau$ to distinct values in 
$\{\tau,\omega\tau,\omega^2\tau\}$. Here, we introduce a function 
that does just that.

\smallskip

Given $(\theta_1,\theta_2)\in \mathsf T$, put (see \cite[Subsection~3.3.3]{Will2017})
$$E_{\theta_1,\theta_2}:=\left[
\begin{matrix}
e^{\frac{2\theta_1-\theta_2}{3}i} & 0 & 0 \\
0 & e^{\frac{2\theta_2-\theta_1}{3}i} & 0 \\
0 & 0 & e^{-\frac{\theta_1+\theta_2}{3}i}
\end{matrix}
\right],$$
and define $\alttrace:\mathsf T\to\mathbb C$ by
$\alttrace(\theta_1,\theta_2)=\trace E_{\theta_1,\theta_2}$.
This function does not descend to a well defined function on $\mathcal E$
since $\trace E_{(\theta,0)}\neq\trace E_{(2\pi,\theta)}$ for every
$0\leq\theta\leq 2\pi$, but it clearly well defines a function
$\alttrace:\mathcal E^{\mathrm{reg}}\to\mathbb C$. Abusing notation, we also evaluate 
$\alttrace$ directly on regular elliptic isometries
by considering $\alttrace F=\alttrace(\theta_1,\theta_2)$,
where $(\theta_1,\theta_2)\in\mathcal E^{\mathrm{reg}}$ is the angle pair of $F$.

\begin{figure}[!ht]
\centering
\includegraphics[scale=1]{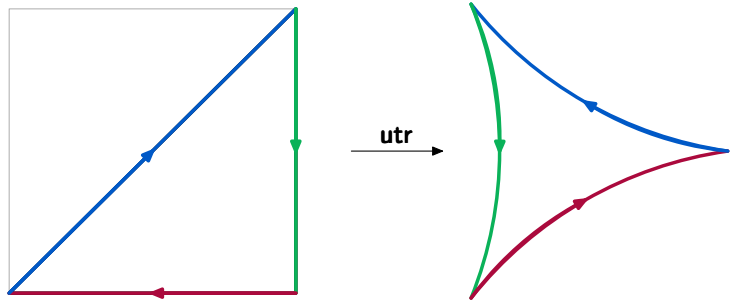}
\caption{How $\alttrace$ maps the boundary of $\mathsf T$ onto $\partial\Delta$. (Colors are
visible in the online version)}
\label{fig:alttraceborder}
\end{figure}

\begin{prop}
\label{prop:alttraceinj}
The function\/ $\alttrace$ maps\/ $\mathsf T$ bijectively onto\/ $\Delta$.
\end{prop}

\begin{proof}
Take $(\theta_1,\theta_2)\in\mathsf T$ and let $F_1\in\SU(2,1)$ be an elliptic
isometry with angle pair $(\theta_1,\theta_2)$ and eigenvalues 
$e^{-\frac{\theta_1+\theta_2}{3}i},e^{\frac{2\theta_1-\theta_2}{3}i},
e^{\frac{2\theta_2-\theta_1}{3}i}$, with $e^{-\frac{\theta_1+\theta_2}{3}i}$ being of
negative type. 

Now, consider $F_2\in\SU(2,1)$ such that $\trace F_1=\trace F_2$ (in particular
$F_1$ and $F_2$, have the same eigenvalues, see~\cite[Proof of Lemma~6.2.5]{Goldman1999}) and such that 
$e^{\frac{2\theta_1-\theta_2}{3}i}$ is its negative type eigenvalue. It follows
that the angle pair of~$F_2$ is $\{-\theta_1,-\theta_1+\theta_2\}$. Using the fact 
that~$0\leq\theta_2\leq\theta_1\leq2\pi$, we obtain that such angle pair projects to
the point $(2\pi-\theta_1+\theta_2,2\pi-\theta_1)\in\mathsf T$. This implies
that $\alttrace F_2=\alttrace(2\pi-\theta_1,2\pi-\theta_1+\theta_2)=\omega\alttrace F_1$.
Therefore, if $\trace F_1\neq 0$, $\alttrace F_1\neq\alttrace F_2$. The argument follows 
analogously if we assume that the negative type eigenvalue of $F_2$ is 
$e^{\frac{2\theta_2-\theta_1}3i}$.

It remains to prove that $\alttrace$ is surjective. Given $\tau\in\Delta^{\circ}$,
$\tau\neq 0$, there are three distinct $\PU(2,1)$-conjugacy classes that admit 
a lift with trace $\tau$ (see Section~\ref{sec:conjclasses}).
Since $\alttrace$ is injective, it sends each of these classes to a distinct
value, and the values it can assume lie in the set $\{\tau,\omega\tau,\omega^2\tau\}$.
For~$\tau=0$, there is only one conjugacy class with such trace. Finally,
$\alttrace$ sends the boundary of $\mathsf T$ onto $\partial\Delta$. In fact, 
$$\alttrace(\theta,\theta)=2e^{\frac{\theta}3i}+e^{-\frac{2\theta}3i},
\quad\alttrace(0,\theta)=2e^{-\frac{\theta}3i}+e^{\frac{2\theta}{3}i},
\quad\text{and}\quad \alttrace(\theta,0)=2e^{-\frac{\theta}3i}+e^{\frac{2\theta}{3}i},$$
for $\theta\in[0,2\pi]$. It follows that each side of the boundary of $\mathsf T$
is mapped onto a distinct side of $\partial\Delta$ (see Figure~\ref{fig:alttraceborder}).
\end{proof}

\begin{cor}
Let\/ $F_1,F_2\in\PU(2,1)$ be regular elliptic isometries. If\/ $\alttrace F_1=
\alttrace F_2$, then\/ $F_1$ and\/ $F_2$ are\/ $\PU(2,1)$-conjugated.
\end{cor}

\subsection{Lines tangent to the deltoid}
\label{subsec:tangentlines}

Given $\alpha_1,\alpha_2\in\mathbb S^1$, consider the function
$\tau_{\alpha_1,\alpha_2}:\mathbb R\to\mathbb C$
defined by
\begin{equation}
\label{eq:traceparam}
\tau_{\alpha_1,\alpha_2}(t):=
\alpha_1\alpha_2+\alpha_1^{-2}\alpha_2+\alpha_1\alpha_2^{-2}+
(\alpha_1^{-2}-\alpha_1)(\alpha_2^{-2}-\alpha_2)t.
\end{equation}
From~\cite[Section~6]{spell} we have the following lemma.

\begin{lemma}
\label{lemma:tracetau}
If\/ $p_1,p_2\in\PV\setminus\SV$ are such that\/
$\tance(p_1,p_2)=t$ {\rm(}see {\rm Section~\ref{sec:hypgeo})}, then
$\trace R_{\alpha_2}^{p_2}R_{\alpha_1}^{p_1}=\tau_{\alpha_1,\alpha_2}(t)$.
\end{lemma}

Moreover, by~\cite[Lemma~6.17]{spell}, $\tau_{\alpha_1,\alpha_2}(\mathbb R)$ is
a line tangent to $\partial\Delta$ at $\tau_{\alpha_1,\alpha_2}(1)$, and it only
depends on the unit complex number $\alpha_1\alpha_2$. We denote this line by 
$\ell_\alpha$ where $\alpha=\alpha_1\alpha_2$.

\begin{prop}
\label{prop:traceeigen}
Let\/ $F\in\SU(2,1)$ and let\/ $\alpha\in\mathbb S^1$. Then\/ $\trace F\in\ell_\alpha$
iff\/ $\alpha$ is an eigenvalue of\/ $F$.
\end{prop}

\begin{proof}
Suppose that $\trace F\in\ell_{\alpha}$, and let $\alpha_1,\alpha_2\in\mathbb S^1
\setminus\Omega$ be such that $\alpha_1\alpha_2=\alpha$. Then, 
$\ell_\alpha=\ell_{\alpha_1\alpha_2}$ which implies that there exists 
$t\in\mathbb R$ such that $\tau_{\alpha_1,\alpha_2}(t)=\trace F$.
Taking $p_1,p_2\in\PV\setminus\SV$ with $\tance(p_1,p_2)=t$, by 
Lemma~\ref{lemma:tracetau}, we have 
$\trace R_{\alpha_2}^{p_2}R_{\alpha_1}^{p_1}=\trace F$ and, since trace determines
eigenvalues and $\alpha_1\alpha_2$ is an eigenvalue of~$R_{\alpha_2}^{p_2}R_{\alpha_1}^{p_1}$,
$\alpha$ is an eigenvalue of $F$.

Conversely, suppose that $\alpha\in\mathbb S^1$ is an eigenvalue of $F$.

Assume that $F$ is not loxodromic.
Then $\trace F=\alpha+\beta+\overline\alpha\overline\beta$ for some 
$\beta\in\mathbb S^1$. In fact, if $F$ is not loxodromic, then it can be triangularized
with unit norm eigenvalues (see \cite[Section~3.2]{Parker2012}). 
Take $\alpha_1,\alpha_2\in\mathbb S^1\setminus\Omega$ 
with $\alpha_1\alpha_2=\alpha$. Note that,
$$z=\frac{\beta+\alpha_1^{-1}\alpha_2^{-1}\beta^{-1}-
\alpha_1^{-2}\alpha_2-\alpha_1\alpha_2^{-2}}{(\alpha_1^{-2}-\alpha_1)
(\alpha_2^{-2}-\alpha_2)}\in\mathbb R.$$
In fact, it is easy to see that $z-\overline z=0$.
Hence, the equation $\tau_{\alpha_1,\alpha_2}(t)=\trace F$ has a real
solution in~$t$; therefore $\trace F\in\ell_{\alpha_1\alpha_2}=\ell_\alpha$.

Now, assume that $F$ is loxodromic. Let $\ell_\beta$ be a line tangent to the
deltoid and passing through~$\trace F$. Thus, given parameters 
$\alpha_1,\alpha_2\in\mathbb S^1\setminus\Omega$ with $\alpha_1\alpha_2=\beta$,
there exist points $p_1,p_2\in\PV\setminus\SV$ such that 
$\trace R_{\alpha_2}^{p_2}R_{\alpha_1}^{p_1}=\trace F$, 
and $\beta$ is an eigenvalue of $F$. Since a loxodromic isometry
cannot have two unit norm eigenvalues (see~\cite[Lemma~3.2]{Parker2012}), we conclude that $\alpha=\beta$ and
$\trace F\in\ell_{\alpha}$.
\end{proof}

It follows that, given $\tau\in\Delta^\circ$, there exists exactly three lines 
through $\tau$ that are tangent to $\partial\Delta$. Writing these lines as
$\ell_\alpha$, $\ell_\beta$, $\ell_{\gamma}$, we have 
$\gamma=\overline\alpha\overline\beta$,  
$\tau=\alpha+\beta+\overline\alpha\overline\beta$ and the eigenvalues
of any isometry with trace $\tau$ are $\alpha,\beta,\overline\alpha\overline\beta$. 
In the case where $\tau\in\partial\Delta$, there are exactly two lines tangent to 
$\partial\Delta$ that contain\/ $\tau$; if one of the lines is $\ell_\alpha$, 
the other is $\ell_{\alpha^{-2}}$.

Furthermore, two distinct lines $\ell_\alpha$ and $\ell_\beta$ intersect in $\Delta$.
In fact, they cannot be parallel since there exists an isometry with trace
$\tau:=\alpha+\beta+\overline\alpha\overline\beta$ and 
$\tau\in\ell_{\alpha}\cap\ell_{\beta}$ by the previous proposition.
Moreover, if they intersect at a point $\tau\in\mathbb C\setminus\Delta$,
then $\tau$ is the trace of a loxodromic isometry that have 
$\alpha$ and $\beta$ as eigenvalues, what cannot happen as loxodromic isometries have a
single unitary eigenvalue.

\subsection{Tangent lines and the unfolded trace}
Here we describe the inverse image under $\alttrace$ of lines tangent to the deltoid;
it will be composed by line segments of slopes $-1,\frac{1}2,2$ in $\mathsf T$.

\begin{lemma}
\label{lemma:walls}
Given\/ $\alpha\in\mathbb S^1$, the inverse image\/
$\alttrace^{-1}(\ell_\alpha\cup\ell_{\omega\alpha}\cup\ell_{\omega^2\alpha})$
is given by the projection on\/ $\mathsf T$ of the lines in\/
$\mathbb R^2$ defined by the equations\/ $y=-x-3a$ and\/ $y=2x-3a$, where\/
$0\leq a<2\pi$ is such that\ $\alpha=e^{ai}$. More precisely,
$\alttrace^{-1}(\ell_\alpha\cup\ell_{\omega\alpha}\cup\ell_{\omega^2\alpha})$
is given by the union of the segments

\smallskip

$\bullet$ $\big[\big(\frac{3a}2,0\big),(3a,3a)\big]$, $\big[(3a,3a),
\big(2\pi,\pi+\frac{3a}2\big)\big]$, $\big[\big(2\pi-\frac{3a}2,2\pi-\frac{3a}2\big),
(2\pi,2\pi-3a)\big]$, $\big[(2\pi,2\pi-3a),\big(\pi+\frac{3a}2,0\big)\big]$,
$\big[\big(\pi-\frac{3a}2,\pi-\frac{3a}2\big),(2\pi-3a,0)\big]$,
$\big[(2\pi-3a,0),\big(2\pi,\frac{3a}2\big)\big]$,
if\/ $0\leq a < \frac{2\pi}3$;

\smallskip

$\bullet$ $\big[(3\pi-\frac{3a}2,3\pi-\frac{3a}2),(2\pi,4\pi-3a)\big]$,
$\big[(2\pi,4\pi-3a),(\frac{3a}2,0)\big]$,
$\big[(2\pi,\frac{3a}2-\pi),(4\pi-3a,0)\big]$,
$\big[(4\pi-3a,0),(2\pi-\frac{3a}2,2\pi-\frac{3a}2)\big]$,
$\big[(\frac{3a}2-\pi,0),(3a-2\pi,3a-2\pi)\big]$,
$\big[(3a-2\pi,3a-2\pi),(2\pi,\frac{3a}2)\big]$,
if\/ $\frac{2\pi}3\leq a < \frac{4\pi}3$;

\smallskip

$\bullet$ $\big[(2\pi,\frac{3a}2-2\pi),(6\pi-3a,0)\big]$,
$\big[(6\pi-3a,0),(3\pi-\frac{3a}2,3\pi-\frac{3a}2)\big]$,
$\big[(\frac{3a}2-2\pi,0),(3a-4\pi,3a-4\pi)\big]$,
$\big[(3a-4\pi,3a-4\pi),(2\pi,\frac{3a}2-\pi)\big]$,
$\big[(4\pi-\frac{3a}2,4\pi-\frac{3a}2),(2\pi,6\pi-3a)\big]$,
$\big[(2\pi,6\pi-3a),(\frac{3a}2-\pi,0)\big]$,
if\/ $\frac{4\pi}3\leq a <2\pi$.
\end{lemma}

\begin{proof}
Let $\tau\in\ell_\alpha$. By Proposition~\ref{prop:traceeigen}, every isometry
with trace $\tau$ has $\alpha$ as an eigenvalue. We write $\alpha=e^{ai}$,
$0\leq a<2\pi$, and
write the other eigenvalues as $e^{ti}$ and $e^{-(a+t)i}$, $0\leq t<2\pi$;
this is well defined since the trace of an element of $\SU(2,1)$ determines its
eigenvalues. It follows that the (nonoriented) angle pairs of the isometries in 
$\PU(2,1)$ with trace $\tau$ are given 
by $\{t-a,-t-2a\}$, $\{-t+a,-2t-a\}$ and $\{t+2a,2t+a\}$, with $t$ varying.
The projection of this pair in $\mathsf T$ coincides with the projection 
of the lines defined in the proposition.
\end{proof}

\begin{figure}[!ht]
\centering
\includegraphics[scale=1]{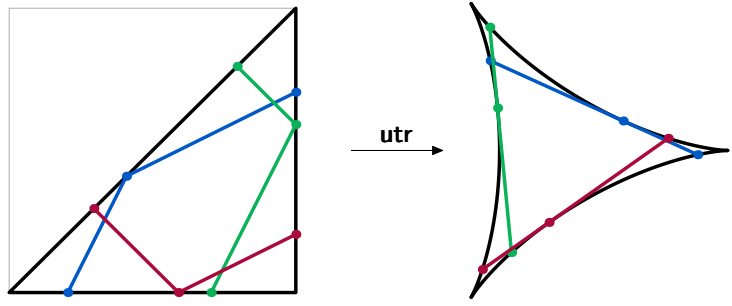}
\caption{Inverse image of lines tangent to $\partial\Delta$. (Colors are
visible in the online version)}
\label{fig:alttracetang}
\end{figure}

\begin{cor}
\label{cor:imageella}
Given\/ $\alpha\in\mathbb S^1$, the inverse image under\/ $\alttrace$ of the line\/ 
$\ell_\alpha$ is given by

\smallskip

$\bullet$ $\big[(\frac{3a}2,0),(3a,3a)\big]\cup\big[(3a,3a),(2\pi,\pi+\frac{3a}2)\big]$,
if\/ $0\leq a<2\pi/3$;

\smallskip

$\bullet$ $\big[(3\pi-\frac{3a}2,3\pi-\frac{3a}2),(2\pi,4\pi-3a)\big]\cup
\big[(2\pi,4\pi-3a),(\frac{3a}2,0)\big]$, if\/ $2\pi/3\leq a<4\pi/3$;

\smallskip

$\bullet$ $\big[(2\pi,\frac{3a}2-2\pi),(6\pi-3a,0)\big]\cup
\big[(6\pi-3a,0),(3\pi-\frac{3a}2,3\pi-\frac{3a}2)\big]$, if\/ $4\pi/3\leq a<2\pi$,

\smallskip

\noindent where $0\leq a<2\pi$ is such that $\alpha=e^{ai}$.
\end{cor}

\begin{proof}
By Lemma~\ref{lemma:walls}, $\alttrace^{-1}(\ell_\alpha\cup\ell_{\omega\alpha}\cup
\ell_{\omega^2\alpha})$
is either the union of three connected curves, each being the union of two line
segments with a vertex in the boundary of $\mathsf T$, if $a\neq 0,2\pi/3,4\pi/3$,
or a median of $\mathsf T$, if $a=0,2\pi/3,4\pi/3$.
Since we know how $\alttrace$ sends the boundary of $\mathsf T$ onto $\partial\Delta$
(see Figure~\ref{fig:alttraceborder}), we determine which of these three curves is 
$\alttrace^{-1}(\ell_\alpha)$.
\end{proof}

Figure~\ref{fig:alttracetang} illustrates Lemma~\ref{lemma:walls} and 
Corollary~\ref{cor:imageella}.
The inverse of $\alttrace$ sends each line tangent to
$\partial\Delta\setminus\{3,3\omega,3\omega^2\}$ to a curve
that is the union of two line segments with a vertex in the boundary
of the triangle $\mathsf T$; it also sends each line tangent to
one of the vertices of $\partial\Delta$ to a median of $\mathsf T$.
The color in Figure~\ref{fig:alttraceborder} shows us how to color
Figure~\ref{fig:alttracetang} (see online version). Pictures relating Goldman's deltoid and 
the triangle $\mathsf T$ have appeared before (see, for example, 
\cite[Figure~1]{PP2009}).

\section{Length 2 decompositions}
\label{sec:length2decomp}

In this section, using the tools introduced in Section~\ref{sec:altracetangent},
we study decompositions of isometries as the product of two
special elliptic isometries. First, in Subsection~\ref{subsec:generic}, 
we prove that every isometry that is not $2$-step unipotent admits
an $(\alpha_1,\alpha_2)$-decomposition, for some parameters $\alpha_1,\alpha_2$.

In Subsection~\ref{subsec:fixedpar}, for given parameters $\alpha_1,\alpha_2$, 
we obtain all isometries that admit an $(\alpha_1,\alpha_2)$-decomposition.
In particular, we determine
which isometries have $\alpha$-length $2$ for a given parameter~$\alpha$.

\subsection{Generic decompositions}
\label{subsec:generic}

The following lemma characterizes the product of two special elliptic isometries
that have their centers generating an Euclidean complex line.

\begin{lemma}
\label{lemma:euclideanline}
Let\/ $p_1,p_2\in\PV\setminus\SV$ be distinct nonisotropic points and let\/
$\alpha_1,\alpha_2\in\mathbb S^1\setminus\Omega$ be parameters. If the complex line\/
$\mathrm L(p_1,p_2)$ is Euclidean, then the isometry\/ 
$R:=R_{\alpha_2}^{p_2}R_{\alpha_1}^{p_1}$ is parabolic. Moreover, either\/
$\alpha_1\alpha_2\notin\Omega$ and\/ $R$ is ellipto-parabolic, or\/ 
$\alpha_1\alpha_2\in\Omega$ and\/ $R$ is\/ $3$-step unipotent.
\end{lemma}

\begin{proof}
By hypothesis, $\tance(p_1,p_2)=1$ (see Section~\ref{sec:hypgeo}) and,
by Lemma~\ref{lemma:tracetau},
$\trace R=\tau_{\alpha_1,\alpha_2}(1)=2\alpha_1\alpha_2+(\alpha_1\alpha_2)^{-2}$. 
Then, the isometry $R$ is either special elliptic or parabolic (not necessarily 
special elliptic since~$p_1\neq p_2$).
Let $v$ be the polar point of the line $\mathrm L(p_1,p_2)$; it follows that
$v$ is an isotropic fixed point of $R$.
Suppose that $R$ stabilizes a hyperbolic line $L$ through $v$. 
Using~\eqref{eq:specialelliptic} to solve the equation 
$R_{\alpha_2}^{p_2}R_{\alpha_1}^{p_1}x=\alpha_1\alpha_2 x$, we obtain that
a point $x\in L$ is $R$-fixed with eigenvalue $\alpha_1\alpha_2$ iff
$\langle x, p_1\rangle=\langle x,p_2\rangle=0$, which implies~$x=v$. Therefore, 
$R$ is parabolic, and it is unipotent iff $\alpha_1\alpha_2=\delta$ for some
$\delta\in\Omega$.
\end{proof}

\begin{rmk}(See~\cite[Section~6]{spell})
\label{rmk:rtilde}
Let $p_1,p_2\in\PV\setminus\SV$ be distinct points such that the line
$L:=\mathrm L(p_1,p_2)$ is hyperbolic and 
let~$\alpha_1,\alpha_2\in\mathbb S^1\setminus\Omega$ be parameters.

Suppose that the isometry $R:=R_{\alpha_2}^{p_2}R_{\alpha_1}^{p_1}$ is regular elliptic.
If $\tilde p_1,\tilde p_2$ are respectively the points in $L$ orthogonal to $p_1,p_2$, then 
$\tance(p_1,p_2)=\tance(\tilde p_1,\tilde p_2)$.
Moreover the isometries $R$ and $\widetilde R:=R_{\alpha_2}^{\tilde p_2}R_{\alpha_1}^{\tilde p_1}$
have the same trace but lie in distinct $\SU(2,1)$-conjugacy classes.

If the isometry $R$ is regular parabolic or loxodromic, an analogous process
produces an 
isometry~$\widetilde R$ that lies in the same $\SU(2,1)$-conjugacy class as $R$. This
implies that there exists a relation of the form $R_{\alpha_2}^{p_2}R_{\alpha_1}^{p_1}
=R_{\alpha_2}^{q_2}R_{\alpha_1}^{q_1}$ with $\sigma p_i=-\sigma q_i$ ($\sigma p$
stands for the signature of a point $p$). Relations obtained in this way are called 
{\it simultaneous change of signs}. 
\end{rmk}

\begin{prop}
Every isometry that is not\/ $2$-step unipotent admits an\/ 
$(\alpha_1,\alpha_2)$-decomposition, for some parameters\/ 
$\alpha_1,\alpha_2\in\mathbb S^1\setminus\Omega$. Moreover, $2$-step unipotent
isometries do not admit an $(\alpha_1,\alpha_2)$-decomposition, for any parameters
$\alpha_1,\alpha_2\in\mathbb S^1\setminus\Omega$.
\end{prop}

\begin{proof}
Since $IR_{\alpha}^p I^{-1}=R_\alpha^{Ip}$ for any
isometry $I\in\SU(2,1)$, in order to prove that a given isometry $F\in\SU(2,1)$
admits an $(\alpha_1,\alpha_2)$-decomposition, it suffices to prove that $F$ 
is in the same $\SU(2,1)$-conjugacy class as an isometry the form 
$R_{\alpha_2}^{p_2}R_{\alpha_1}^{p_1}$, for some
$p_1,p_2\in\PV\setminus\SV$ and $\alpha_1,\alpha_2\in\mathbb S^1\setminus\Omega$.

If $F=R_\beta^q$, we just take $\alpha_1,\alpha_2\in\mathbb S^1\setminus\Omega$
with $\alpha_1\alpha_2=\beta$ and we have $F=R_{\alpha_2}^q R_{\alpha_1}^q$. 

Assume that $F$ is loxodromic. Let $\alpha\in\mathbb C$ be the unit norm eigenvalue
of $F$. Then, by Proposition~\ref{prop:traceeigen}, $\trace F\in\ell_\alpha$.
If $\alpha_1,\alpha_2\in\mathbb S^1\setminus\Omega$ are such that $\alpha_1\alpha_2=
\alpha$, then $\ell_\alpha=\ell_{\alpha_1\alpha_2}$ and, by Lemma~\ref{lemma:tracetau}, 
$\trace F=\trace R_{\alpha_2}^{p_2}R_{\alpha_1}^{p_1}$ for
some $p_1,p_2\in\PV\setminus\SV$.
Since two loxodromic isometries in $\SU(2,1)$ with 
the same trace are conjugated (see Section~\ref{sec:conjclasses}), the result follows.

Now suppose that $F$ is regular elliptic; let 
$\alpha$ and $\beta$ be two distinct eigenvalues of $F$ with 
$\alpha,\beta\in\mathbb S^1\setminus\Omega$. Let $\alpha_i,\beta_i
\in\mathbb S^1\setminus(\Omega\cup-\Omega)$ be parameters such that 
$\alpha_1\alpha_2=\alpha$,
$\beta_1\beta_2=\beta$ (here, $-\Omega:=\{-1,-\omega,-\omega^2\}$). 
Then, by Lemma~\ref{lemma:tracetau}, there exist $s,t\in\mathbb R$ satisfying 
$\tau_{\alpha_1,\alpha_2}(s)=\tau_{\beta_1,\beta_2}(t)=\trace F$. Taking
$-\alpha_i$ (resp., $-\beta_i$) in place of $\alpha_i$ (resp., $\beta_i$) if necessary,
we can assume that $s,t\not\in[0,1]$; this follows from the fact that
$\tau_{\alpha_1,\alpha_2}$ and $\tau_{-\alpha_1,-\alpha_2}$ parametrize
$\ell_{\alpha}$ in opposite directions (see~\cite[Lemma~6.19]{spell}). Now, let 
$p_i,q_i\in\PV\setminus\SV$, $i=1,2$, be points with $\tance(p_1,p_2)=s$
and $\tance(q_1,q_2)=t$; it follows that $\trace R_{\alpha_2}^{p_2}
R_{\alpha_1}^{p_1}=\trace R_{\beta_2}^{q_2}R_{\beta_1}^{q_1}$.
Denote by $\tilde p_i$ the point in $L_1:=\mathrm L(p_1,p_2)$ orthogonal to $p_i$,
and by $\tilde q_i$ the point in $L_2:=\mathrm L(q_1,q_2)$ orthogonal to $q_i$. 
Note that since $s,t\not\in[0,1]$, the lines $L_1$ and $L_2$ are hyperbolic, thus
$\sigma\tilde p_i=-\sigma p_i$ and $\sigma\tilde q_i=-\sigma q_i$. 
By Remark~\ref{rmk:rtilde}, $R_1:=R_{\alpha_2}^{p_2}R_{\alpha_1}^{p_1}$ and 
$\widetilde R_1:=R_{\alpha_2}^{\tilde p_2}R_{\alpha_1}^{\tilde p_1}$ are isometries 
of same trace (equal to $\trace F$) but lying in distinct $\SU(2,1)$-conjugacy classes;
the same is true for the isometries $R_2:=R_{\beta_2}^{q_2}R_{\beta_1}^{q_1}$ and 
$\widetilde R_2:=R_{\beta_2}^{\tilde q_2}R_{\beta_1}^{\tilde q_1}$. Therefore,
since we have three $\SU(2,1)$-conjugacy classes for each trace, 
one of the isometries $R_1,\widetilde R_1,R_2,\widetilde R_2$ must lie in the
$\SU(2,1)$-conjugacy class of $F$.

Finally, the case where $F$ is ellipto-parabolic or $3$-step unipotent follows
from Lemma~\ref{lemma:euclideanline}: if $\alpha$ and $\alpha^{-2}$ are the 
eigenvalues of $F$ we just need to take parameters $\alpha_1,\alpha_2\in
\mathbb S^1\setminus\Omega$ with $\alpha_1\alpha_2=\alpha$ and points 
$p_1,p_2\in\PV\setminus\SV$ that generate a Euclidean line. If follows that
$F$ is conjugated to $R_{\alpha_2}^{p_2}R_{\alpha_1}^{p_1}$. (If $F$ is 
ellipto-parabolic, proceding as above and taking $\alpha_1\alpha_2=\alpha^{-2}$, 
we obtain $p_1,p_2$ such that $\mathrm L(p_1,p_2)$ is noneuclidean.)

The second part of the proposition follows from the fact that, if $p_1\neq p_2$
the isometry $R_{\alpha_2}^{p_2}R_{\alpha_1}^{p_1}$ is regular (see 
Definition~\ref{defi:regular}, Lemma~\ref{lemma:euclideanline}, 
and~\cite[Lemma~4.2]{spell}).
\end{proof}

\subsection{Decompositions with fixed parameters}
\label{subsec:fixedpar}

Here we are interested in the case where the parameters $\alpha_1,\alpha_2$ are given.
We start determining which regular parabolic
or loxodromic isometries admit an $(\alpha_1,\alpha_2)$-decomposition. 

\begin{prop}
\label{prop:a1a2loxpardecomp}
Let\/ $\alpha_1,\alpha_2\in\mathbb S^1\setminus\Omega$ be parameters. 
The following statements hold:

$\bullet$ a loxodromic isometry admits an\/ $(\alpha_1,\alpha_2)$-decomposition iff
it can be lifted to an isometry\/ $F\in\SU(2,1)$ with\/ $\trace F\in
\ell_{\alpha_1\alpha_2}$ {\rm(}see\/ {\rm Subsection~\ref{subsec:tangentlines}}{\rm)};

$\bullet$ a regular parabolic isometry\/ {\rm(Definition~\ref{defi:regular})} 
admits an\/ $(\alpha_1,\alpha_2)$-decomposition iff it can be lifted to an isometry\/ 
$F\in\SU(2,1)$ with\/ $\trace F\in\ell_{\alpha_1\alpha_2}$ and\/ 
$\trace F\neq\tau_{\alpha_1,\alpha_2}(0)$.
\end{prop}

\begin{proof}
First, note that, by Proposition~\ref{prop:traceeigen}, if an isometry admits 
an $(\alpha_1,\alpha_2)$-decomposition, then it lifts to $\SU(2,1)$ to an isometry
with trace lying in $\ell_{\alpha_1\alpha_2}$. Moreover, if $p_1,p_2\in\PV\setminus\SV$ 
are orthogonal points (i.e., $\tance(p_1,p_2)=0$), then 
$\delta R_{\alpha_2}^{p_2}R_{\alpha_1}^{p_1}$ is an elliptic isometry, for 
any~$\delta\in\Omega$.

Suppose that $F\in\SU(2,1)$ is loxodromic and $\trace F\in\ell_{\alpha_1\alpha_2}$.
If $t_0\in\mathbb R$ is such that $\trace F=\tau_{\alpha_1,\alpha_2}(t_0)$, and 
$p_1,p_2\in\PV\setminus\SV$ satisfy $\tance(p_1,p_2)=t_0$, then 
$\trace F=\trace R_{\alpha_2}^{p_2}R_{\alpha_1}^{p_1}$ and, since two loxodromic 
isometries with the same trace are $\SU(2,1)$-conjugated, the result follows.

Now, suppose that $F\in\SU(2,1)$ is regular parabolic
with $\trace F\in\ell_{\alpha_1\alpha_2}$ and $\trace F\neq\tau_{\alpha_1,\alpha_2}(0)$.
Parametrize the line $\ell_{\alpha_1\alpha_2}$ by $\tau_{\alpha_1,\alpha_2}(t)$ and let
$t_0\in\mathbb R$ be such that $\tau_{\alpha_1,\alpha_2}(t_0)=\trace F$.
By hypothesis,~$t_0\neq 0$.
If $t_0\neq 1$, then $F$ is conjugated to $R:=R_{\alpha_2}^{p_2}R_{\alpha_1}^{p_1}$,
where $p_1,p_2\in\PV\setminus\SV$ are any points with $\tance(p_1,p_2)=t_0$. 
(This follows from $\trace F=\trace R$ together with the fact that
$R$ is not special elliptic, since~$t_0\neq 0,1$.) 
If~$t_0=1$, let $p_1,p_2\in\EV$ be points such that the line $\mathrm L(p_1,p_2)$
is Euclidean and let $R$ be defined as above. Then, $\trace F=\trace R$ and, 
by Lemma~\ref{lemma:euclideanline}, the isometry $R$ is either ellipto-parabolic or
$3$-step unipotent. Therefore, $R$ and $F$ are conjugated.
\end{proof}

Now, to obtain the remaining isometries (regular or special elliptic)
admitting an $(\alpha_1,\alpha_2)$-decomposition, we determine 
$\mathsf E_{\alpha_1,\alpha_2}$ (see Notation~\ref{notation:egclasses}) and the
classes of its intersection with the boundary of $\rho(\mathcal E)$. 
We write $\mathsf E_{\alpha_1,\alpha_2}$ as the union of sets 
$\mathsf E_{\alpha_1,\alpha_2}^{\sigma_1\sigma_2}$
that are defined as follows.

\begin{defi}
\label{defi:esigma}
We denote by $\mathsf E_{\alpha_1,\alpha_2}^{\sigma_1\sigma_2}$ the set
composed by (the projection on $\rho(\mathcal E)$ of) classes of elliptic or parabolic
isometries admitting a decomposition of the form $F=R_{\alpha_2}^{p_2}R_{\alpha_1}^{p_1}$, 
with $\sigma p_i=\sigma_i$, $i=1,2$.
(To simplify the notation, the signs $\sigma_1,\sigma_2$ are taken as one of the 
symbols $-,+$ instead of values in~$\{-1,1\}$.)
\end{defi}

\begin{prop}
\label{prop:a1a2eldecomp}
Let\/ $\alpha_1=e^{a_1i}$ and\/ $\alpha_2=e^{a_2i}$ be parameters with\/
$0<a_j<2\pi/3$, $j=1,2$. Then\/~$\mathsf E_{\alpha_1,\alpha_2}$ is given by the union
of the sets\/ $\mathsf E_{\alpha_1,\alpha_2}^{\sigma_1\sigma_2}$ where:

\smallskip

$\bullet$ $\mathsf E_{\alpha_1,\alpha_2}^{--}$ is a single line segment\/ 
{\rm(}possibly, a single point\/{\rm)} given either by\/ 
$\big[(3a_1+3a_2,3a_1+3a_2),\big(2\pi,\pi+\frac{3(a_1+a_2)}{2}\big)\big]$, 
if\/ $0<a_1+a_2\leq 2\pi/3$, or by\/
$\big[(3(a_1+a_2)-2\pi,3(a_1+a_2)-2\pi),\big(\frac{3(a_1+a_2)}{2}-\pi,0\big)\big]$, 
if\/~$a_1+a_2>2\pi/3$.
If\/ $\mathsf E_{\alpha_1,\alpha_2}^{--}$ is a single point, then it is the class of the
identical isometry. Otherwise, the vertex of\/ $\mathsf E_{\alpha_1,\alpha_2}^{--}$ 
in the diagonal side of\/ $\rho(\mathcal E)$, corresponds to a special elliptic isometry 
with negative center, and the vertex of\/ $\mathsf E_{\alpha_1,\alpha_2}^{--}$ lying
in the nondiagonal side of\/ $\rho(\mathcal E)$ corresponds to an ellipto-parabolic isometry.

\smallskip

$\bullet$ $\mathsf E_{\alpha_1,\alpha_2}^{++}$ is the union of two line segments with a 
common vertex, one being of slope\/ $-1$ and the other of slope $\frac 12$ or $2$
{\rm(}possibly, a single point{\rm)}, given~by 
$$\big[(2\pi-3a_1,2\pi-3a_2),\big(2\pi,2\pi-3(a_1+a_2)\big)\big]\cup
\Big[(2\pi,2\pi-3(a_1+3a_2)),\big( \pi+\tfrac{3(a_1+a_2)}{2},0\big)\Big],$$
if\/ $a_1\leq a_2$ and\/ $0<a_1+a_2\leq 2\pi/3$; by
$$\big[(2\pi-3a_2,2\pi-3a_1),(2\pi,2\pi-3(a_1+a_2))\big]\cup
\Big[(2\pi,2\pi-3(a_1+a_2)),\big(\pi+\tfrac{3(a_1+a_2)}{2},0\big)\Big],$$
if\/ $a_1\geq a_2$ and\/ $0<a_1+a_2\leq 2\pi/3$; by
$$\Big[\big(2\pi,\tfrac{3(a_1+a_2)}{2}-\pi\big),(4\pi-3(a_1+a_2),0)\Big]\cup
\big[(4\pi-3(a_1+a_2),0),(2\pi-3a_1,2\pi-3a_2)\big],$$
if\/ $a_1\leq a_2$ and\/ $a_1+a_2> 2\pi/3$; or by
$$\Big[\big(2\pi,\tfrac{3(a_1+a_2)}{2}-\pi\big),(4\pi-3(a_1+a_2),0)\Big]\cup
\big[(4\pi-3(a_1+a_2),0),(2\pi-3a_2,2\pi-3a_1)\big],$$
if\/ $a_1\geq a_2$ and\/ $a_1+a_2>2\pi/3$. If the second segment is a single point,
it corresponds to both the identical class and the\/ $3$-step unipotent class.
Otherwise, the common vertex of the segments
corresponds to both a class of a special elliptic isometry with positive center
and an ellipto-parabolic class; the other vertex of the segment of 
slope\/ $-1$ corresponds to a\/ {\rm (}possibly special\/{\rm)} elliptic isometry;
the remaining vertex corresponds to an ellipto-parabolic isometry.

\smallskip

$\bullet$ $\mathsf E_{\alpha_1,\alpha_2}^{+-}$ is a single line segment\/
{\rm (}possibly, a point\/{\rm )} given by\/
$\big[(3a_2,3a_2-3a_1),\big(\frac{3(a_1+a_2)}{2},0\big)\big]$,
if\/ $a_1\leq a_2$; or by\/
$\big[(2\pi+3a_2-3a_1,3a_2),\big(2\pi,\tfrac{3(a_1+a_2)}{2}\big)\big]$,
if\/ $a_1\geq a_2$. Furthermore, $\mathsf E_{\alpha_1,\alpha_2}^{+-}$ is a single point
corresponding to the class of a special elliptic isometry with positive center
iff\/ $\alpha_1=\alpha_2$; otherwise it has a vertex in\/ $\mathcal E^{\text{reg}}$
and the other vertex correspond to an ellipto-parabolic class.

\smallskip

$\bullet$ $\mathsf E_{\alpha_1,\alpha_2}^{-+}$ is a single line segment\/
{\rm (}possibly, a point\/{\rm )} given by\/
$\big[(2\pi+3a_1-3a_2,3a_1),\big(2\pi,\frac{3(a_1+a_2)}{2}\big)\big]$,
if\/ $a_1\leq a_2$, or by\/
$\big[(3a_1,3a_1-3a_2),\big(\frac{3(a_1+a_2)}{2},0\big)\big]$,
if\/ $a_1\geq a_2$. Furthermore, $\mathsf E_{\alpha_1,\alpha_2}^{-+}$ is a single point
corresponding to the class of a special elliptic isometry with positive center
iff\/ $\alpha_1=\alpha_2$; otherwise it has a vertex in\/ $\mathcal E^{\text{reg}}$
and the other vertex correspond to an ellipto-parabolic class. 
\end{prop}

\begin{figure}[!ht]
\centering
\includegraphics{pics/abdecomp.mps}
\caption{The lift to $\mathsf T$ of the set $\mathsf E_{\alpha_1,\alpha_2}$ and 
its image under $\alttrace$ for parameters 
$\alpha_1=e^{\frac{\pi}{10}i}$ and $\alpha_2=e^{\frac{\pi}4i}$.
The lifts of $\mathsf E_{\alpha_1,\alpha_2}^{--}$, 
$\mathsf E_{\alpha_1,\alpha_2}^{++}$, $\mathsf E_{\alpha_1,\alpha_2}^{+-}$,
and $\mathsf E_{\alpha_1,\alpha_2}^{-+}$ are in blue, green, pink, and orange,
respectively (in the online version). The solid points correspond to elliptic or 
identical classes, while 
the punctured ones correspond to regular parabolic classes.}
\label{fig:abdecomptr}
\end{figure}

\begin{proof}
Denote $\beta_j:=\omega^{j-1}\alpha_1$, $j=1,2,3$.
By Proposition~\ref{prop:traceeigen}, 
$\mathsf E_{\alpha_1,\alpha_2}$ is contained in the subset of $\rho(\mathcal E)$ determined by 
$\alttrace^{-1}(\ell_{\beta_1\alpha_2}\cup\ell_{\beta_2\alpha_2}\cup
\ell_{\beta_3\alpha_2})\subset\mathsf T$.
We consider the lines $\ell_{\beta_j\alpha_2}$ parametrized by
$\tau_{\beta_j,\alpha_2}(t)$, $j=1,2,3$, (see 
Subsection~\ref{subsec:tangentlines}). Define 
$\chi_1:=\imag\big(\frac{\alpha_1}{\alpha_1^{-2}-\alpha_1})$,
$\chi_2:=\imag\big(\frac{\alpha_2}{\alpha_2^{-2}-\alpha_2})$,
and
$$t_{\pm}:=\frac{1+4\chi_1\chi_2\pm\sqrt{(1+4\chi_1^2)(1+4\chi_2^2)}}{2}.$$
By \cite[Corollary~5.8]{spell}, $\tau_{\beta_j,\alpha_2}(t_\pm)$
lie in $\partial\Delta$, and $t_-\leq 0< 1\leq t_+$. Moreover, for~$j=1,2,3$,

\smallskip

({\it A\/}) $\tau_{\beta_j,\alpha_2}(t_-)=
\omega^{j-1}\big(2e^{-\frac{a_1+a_2}{2}i}+e^{(a_1+a_2)i}\big)$;

\smallskip

({\it B\/}) $\tau_{\beta_j,\alpha_2}(0)
=\omega^{j-1}\big(e^{(a_1+a_2)i}+e^{(-2a_1+a_2)i}+e^{(a_1-2a_2)i}\big)$;

\smallskip

({\it C\/}) $\tau_{\beta_j,\alpha_2}(1)=
\omega^{j-1}\big(2e^{(a_1+a_2)i}+e^{-(2a_1+2a_2)i}\big)$;

\smallskip

({\it D\/}) $\tau_{\beta_j,\alpha_2}(t_+)=
\omega^{j-1}\big(-2e^{-\frac{a_1+a_2}{2}i}+e^{(a_1+a_2)i}\big)$.

\smallskip
 
Note that the points in ({\it C\/}) are the ones where the lines 
$\ell_{\beta_j,\alpha_2}$ are tangent to $\partial\Delta$, and that
if $\alpha_1=\alpha_2$, then $t_-=0$ and the points in ({\it A}\/) coincide
with the one in ({\it B}\/).

\smallskip

($\mathsf E_{\alpha_1,\alpha_2}^{--}$)
Let $p_1\in\BV$ and let $L$ be a hyperbolic complex line through
$p_1$. Consider a curve $\gamma:[1,t_+]\to\BV$ such that
$\gamma(t)\in L$ and  $\tance(p_1,\gamma(t))=t$ for all
$t\in[1,t_+]$; in particular $\gamma(1)=p_1$.
Note that every elliptic or regular parabolic isometry $F\in\SU(2,1)$ that admits a
decomposition $R_{\alpha_2}^{q_2}R_{\alpha_1}^{q_1}$ with
$q_1,q_2\in\BV$ is in the same $\SU(2,1)$-conjugacy class of
the isometry $R_{\alpha_2}^{\gamma(\tance(q_1,q_2))}R_{\alpha_1}^{p_1}$.
Thus, the curve $\widetilde\gamma$ in $\mathcal G$ defined by 
$\widetilde\gamma(t)=[R_{\alpha_2}^{\gamma(t)}
R_{\alpha_1}^{p_1}]$ is such that the projection on $c(\mathcal G)$ 
of its image is $\mathsf E_{\alpha_1,\alpha_2}^{--}$.
Moreover, $R_{\alpha_2}^{\gamma(1)}R_{\alpha_1}^{p_1}=R_{\alpha_1\alpha_2}^{p_1}$ 
is special elliptic with angle pair $\{3a_1+3a_2,3a_1+3a_2\}$, and 
$R_{\alpha_2}^{\gamma(t_+)}R_{\alpha_1}^{p_1}$ is parabolic with one of the numbers 
in ({\it D}\/) as trace, which implies that its angle pair is
$\{\pi+\frac{3(a_1+a_2)}2,0\}$. 
 
\smallskip

($\mathsf E_{\alpha_1,\alpha_2}^{++}$) Let $p_1\in\EV$ and consider
a spherical line $L_1$ and a hyperbolic line $L_2$ both through~$p_1$.
Consider a (continuous) curve $\gamma:[0,t_+]\to\EV$
such that $\gamma(t)\in L_1$ for every $t\in[0,1]$,
$\gamma(t)\in L_2$ for every $[1,t_+]$, and
$\tance(p_1,\gamma(t))=t$ for every $t$; in particular $\gamma(1)=p_1$.
Clearly, every elliptic isometry of the form
$R_{\alpha_2}^{q_2}R_{\alpha_1}^{q_1}$, with $q_1,q_2\in\EV$,
is $\SU(2,1)$-conjugated to the isometry $R_{\alpha_2}^{\gamma(\tance(q_1,q_2))}
R_{\alpha_1}^{p_1}$. Moreover, if $q_1,q_2\in\EV$ are such that 
the line $\mathrm L(q_1,q_2)$ is Euclidean, then $R_{\alpha_2}^{q_2}R_{\alpha_1}^{q_1}$ 
is regular parabolic with the same angle pair as $R_{\alpha_1\alpha_2}^{p_1}$, and if
$q_1,q_2\in\EV$ are distinct points such that $R_{\alpha_2}^{q_2}R_{\alpha_1}^{q_1}$
is parabolic, then such isometry is conjugated to 
$R_{\alpha_2}^{\gamma(t_+)}R_{\alpha_1}^{p_1}$. 
Therefore, the projection on $c(\mathcal G)$ of the 
image of the curve $\widetilde\gamma:[0,t_+]\to\mathcal G$, defined as in the previous case, is 
$\mathsf E_{\alpha_1,\alpha_2}^{++}$. The result follows from the fact that
$R_{\alpha_2}^{\gamma(0)}R_{\alpha_1}^{p_1}$ 
is an elliptic isometry with angle pair $\{-3a_1,-3a_2\}$; and 
$R_{\alpha_2}^{\gamma(1)}R_{\alpha_1}^{p_1}=R_{\alpha_1\alpha_2}^{p_1}$
is a special elliptic isometry with angle pair $\{-3(a_1+a_2),0\}$, 
and $R_{\alpha_2}^{\gamma(t_+)}R_{\alpha_1}^{p_1}$ is a parabolic isometry
having one of the points in $({\it D}\/)$ as trace, which implies that 
its angle pair is $\big\{\pi+\frac{3(a_1+a_2)}{2},0\big\}$.

\smallskip

($\mathsf E_{\alpha_1,\alpha_2}^{+-}$) Let $p_1\in\EV$ and let $L$ be
a hyperbolic line through $p_1$. Consider a curve $\gamma:[t_-,0]\to\BV$
such that $\gamma(t)\in L$ and $\tance(p_1,\gamma(t))=t$ for all $t\in[t_-,0]$.
In particular $\langle p_1,\gamma(0)\rangle=0$, i.e., $\gamma(0)$ is the
point in $L$ orthogonal to $p_1$. As before, this defines a curve 
$\widetilde\gamma$ in $\mathcal G$ whose image, projected on $c(\mathcal G)$,
is $\mathsf E_{\alpha_1,\alpha_2}^{+-}$. Note that
$R_{\alpha_2}^{\gamma(t_-)}R_{\alpha_1}^{p_1}$ is a parabolic 
(if $\alpha_1\neq\alpha_2$) or special elliptic (if $\alpha_1=\alpha_2$)
isometry with trace being one of the points in ({\it A}\/) (that coincides
with ({\it B}\/) if $\alpha_1=\alpha_2$), which implies that its angle
pair is $\big\{\frac{3(a_1+a_2)}2,0\big\}$.
Also, 
$R_{\alpha_2}^{\gamma(0)}R_{\alpha_1}^{p_1}$ is an elliptic isometry with
angle pair $\{3a_2,3a_2-3a_1\}$, and the result follows.

($\mathsf E_{\alpha_1,\alpha_2}^{-+}$) This case is analogous to the one above,
considering $p_1\in\BV$ and $\gamma$ as a curve in~$\EV$.  
\end{proof}

\begin{rmk}
\label{rmk:a1a2props}
The following observations are direct consequences of Proposition~\ref{prop:a1a2eldecomp}
and its proof. Figure~\ref{fig:abdecomptr} might be useful as an illustration.
\begin{enumerate}[(1)]
\item $\mathsf E_{\alpha_1,\alpha_2}^{++}$ is the only subset
of $\mathsf E_{\alpha_1,\alpha_2}$ containing a segment of slope $-1$,
and it is composed by classes of elliptic isometries that admit
a decomposition $F=R_{\alpha_2}^{p_2}R_{\alpha_1}^{p_1}$ with spherical line
$\mathrm L(p_1,p_2)$.
\item\label{item:complement}
Define $\mathsf R_{\alpha_1,\alpha_2}^{\sigma_1\sigma_2}:=
\mathsf E_{\alpha_1,\alpha_2}^{\sigma_1\sigma_2}\cap\mathcal E^{\mathrm{reg}}$ (see 
Subsection~\ref{subsec:puclasses}). There exist
$\delta_0,\delta_1,\delta_2\in\Omega$ such that 
$\alttrace(\mathsf R_{\alpha_1,\alpha_2}^{++})\subset
\ell_{\delta_0\alpha_1\alpha_2}$,
$$\big(\delta_1\alttrace(\mathsf R_{\alpha_1,\alpha_2}^{+-})\big)\cup
\alttrace(\mathsf R_{\alpha_1,\alpha_2}^{++})=\ell_{\delta_0\alpha_1\alpha_2}\cap
\Delta^\circ\quad \text{and}\quad
\big(\delta_2\alttrace(\mathsf R_{\alpha_1,\alpha_2}^{-+})\big)\cup
\alttrace(\mathsf R_{\alpha_1,\alpha_2}^{++})=\ell_{\delta_0\alpha_1\alpha_2}\cap
\Delta^\circ.$$
\item By item~({\it iii}\/) of Proposition~\ref{prop:decomp}, 
$\mathsf E_{\alpha_1,\alpha_2}=\mathsf E_{\alpha_2,\alpha_1}$. Moreover,
since 
$R_{\alpha_2}^{p_2}R_{\alpha_1}^{p_1}=R_{\alpha_1}^{R_{\alpha_2}^{p_2}p_1}
R_{\alpha_2}^{p_2}$,
for any points $p_1,p_2\in\PV\setminus\SV$,
we have $\mathsf E_{\alpha_1,\alpha_2}^{\sigma_1\sigma_2}=
\mathsf E_{\alpha_2,\alpha_1}^{\sigma_2\sigma_1}$. 
Therefore, if $\alpha_1=\alpha_2=:\alpha$, then 
$\mathsf E_{\alpha,\alpha}^{+-}=\mathsf E_{\alpha,\alpha}^{-+}$
is a single point in a nondiagonal side of
$\rho(\mathcal E)$. It follows from item~(2) above that there exists 
$\delta\in\Omega$ such that $\alttrace(\mathsf R_{\alpha,\alpha}^{++})=
\ell_{\delta\alpha^2}\cap\Delta^\circ$.
\end{enumerate}
\end{rmk}

\begin{figure}[!htb]
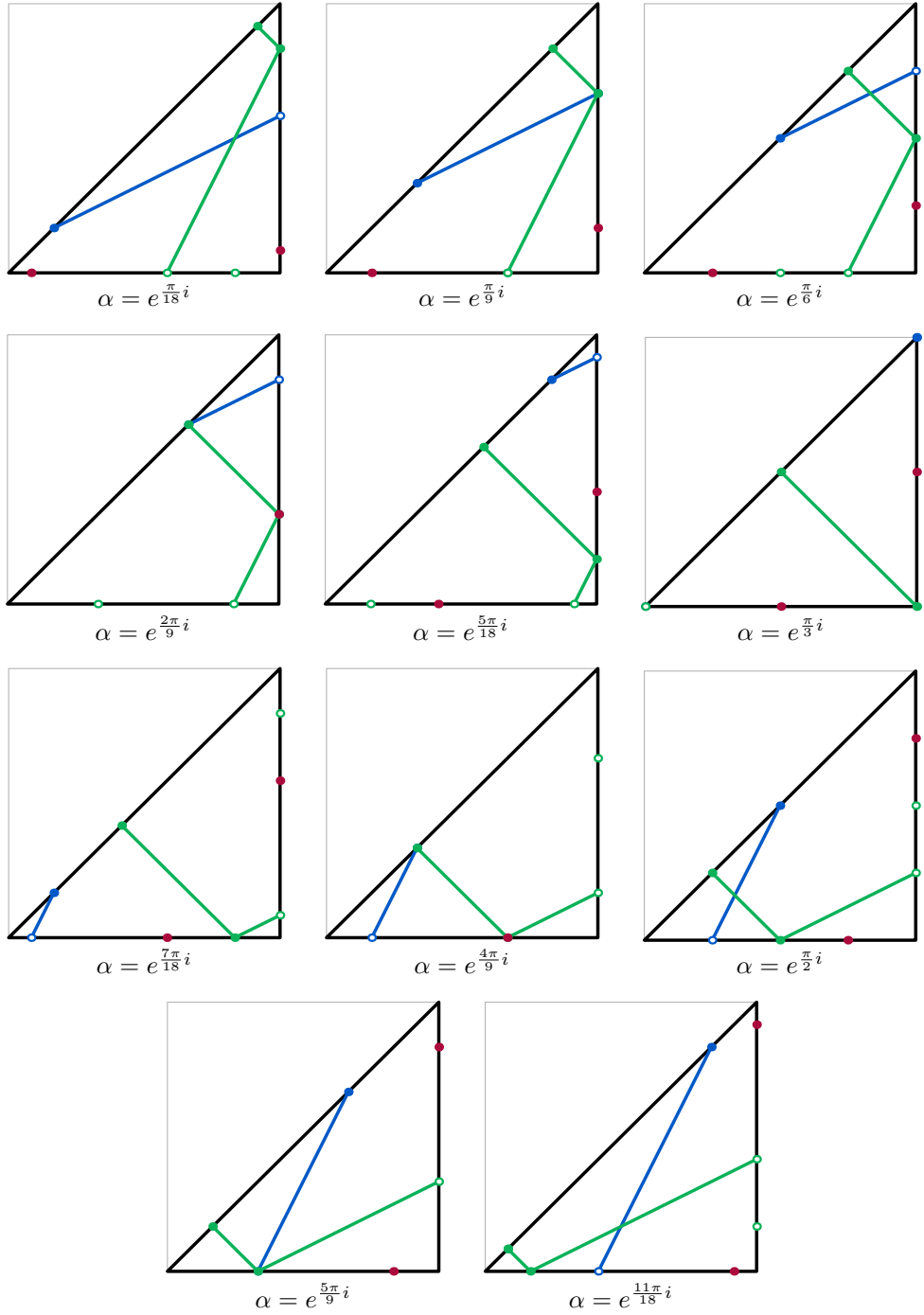

\centering
\includegraphics[scale=1]{pics/a2decomp01.mps}
\hspace{.3cm}
\includegraphics[scale=1]{pics/a2decomp02.mps}
\hspace{.3cm}
\includegraphics[scale=1]{pics/a2decomp03.mps}

\vspace{.3cm}

\includegraphics[scale=1]{pics/a2decomp04.mps}
\hspace{.3cm}
\includegraphics[scale=1]{pics/a2decomp05.mps}
\hspace{.3cm}
\includegraphics[scale=1]{pics/a2decomp06.mps}

\vspace{.3cm}

\includegraphics[scale=1]{pics/a2decomp07.mps}
\hspace{.3cm}
\includegraphics[scale=1]{pics/a2decomp08.mps}
\hspace{.3cm}
\includegraphics[scale=1]{pics/a2decomp09.mps}

\vspace{.3cm}

\includegraphics[scale=1]{pics/a2decomp10.mps}
\hspace{.3cm}
\includegraphics[scale=1]{pics/a2decomp11.mps}
\caption{The set $\mathsf E_{\alpha^{(2)}}$ for given values of $\alpha$. 
The subset $\mathsf E_{\alpha,\alpha}^{--}$ is in blue,
$\mathsf E_{\alpha,\alpha}^{++}$ is in green, and
$\mathsf E_{\alpha,\alpha}^{+-}=\mathsf E_{\alpha,\alpha}^{-+}$ is in red 
(in the online version).
Solid points correspond to elliptic classes, while punctured ones correspond 
to parabolic classes.}
\label{fig:alpha2decomp}
\end{figure}

\begin{cor}
Let\/ $\alpha:=e^{ai}$ be a parameter with\/ $0<a<2\pi/3$.
Then\/ $\mathsf E_{\alpha^{(2)}}$ is given as the union of its subsets\/
$\mathsf E_{\alpha,\alpha}^{\sigma_1\sigma_2}$ where

\smallskip

$\bullet$ $\mathsf E_{\alpha,\alpha}^{--}=\big[(6a,6a),(2\pi,\pi+3a)\big]$ if\/ 
$0<a\leq\pi/3$, or\/ $\mathsf E_{\alpha,\alpha}^{--}=
\big[(6a-2\pi,6a-2\pi),(3a-\pi,0)\big]$ if\/ $\pi/3< a<2\pi/3$;

\smallskip

$\bullet$ $\mathsf E_{\alpha,\alpha}^{++}=\big[(2\pi-3a,2\pi-3a),(2\pi,2\pi-6a)\big]
\cup\big[(2\pi,2\pi-6a),(\pi+3a,0)\big]$ if\/ $0<a\leq\pi/3$, or\/
$\mathsf E_{\alpha,\alpha}^{++}=\big[(2\pi,3a-\pi),(4\pi-6a,0)\big]\cup
\big[(4\pi-6a,0),(2\pi-3a,2\pi-3a)\big]$, if\/ $\pi/3< a<2\pi/3$;

\smallskip

$\bullet$ $\mathsf E_{\alpha,\alpha}^{+-}=\mathsf E_{\alpha,\alpha}^{-+}=
(3a,0)\simeq(2\pi,3a)$.
\end{cor}

Figure~\ref{fig:alpha2decomp} illustrates the set $\mathsf E_{\alpha^{(2)}}$
for some parameters $\alpha$. Summarizing, the space of classes with $\alpha$-length
equal to $2$ is given by the loxodromic isometries that have a lift with trace lying 
in~$\ell_{\alpha^2}$, and by the space $\mathsf E_{\alpha^{(2)}}$, considering the 
description of the classes of its intersection with the boundary of $\rho(\mathcal E)$
obtained in Proposition~\ref{prop:a1a2eldecomp}.

\section{Length 3 decomposition}

In this section we describe all isometries in $\PU(2,1)$,
that are not $2$-step unipotent,
admitting an $\alpha^{(3)}$-decomposition, i.e., those with 
$\alpha$-length equal to $3$, for a given parameter~$\alpha$. We start by proving 
that all loxodromic and regular parabolic isometries admit an 
$\alpha^{(3)}$-decomposition,
for any parameter $\alpha\in\mathbb S^1\setminus\Omega$. Then, using
the results of Section~\ref{sec:length2decomp}, we obtain that this is also true 
for special elliptic isometries with positive center.

In order to obtain the remaining isometries admitting an $\alpha^{(3)}$-decomposition, we
describe $\mathsf E_{\alpha^{(3)}}$ (see Notation~\ref{notation:egclasses}) using the
properties of the product map $\overline\mu$, that are summarized
in Subsection~\ref{subsec:productmap}. The set $\mathsf E_{\alpha^{(3)}}$ is the
union of closed chambers in $\rho(\mathcal E)$ delimited by the 
inverse image under $\alttrace$ of the tangent lines 
$\ell_{\alpha^3},\ell_{\omega\alpha^3},\ell_{\omega^2\alpha^3}$; each of these chambers is either 
full or empty (Proposition~\ref{prop:a3walls}). In Subsection~\ref{subsec:fullempty},
we decide whether each such chamber is full or empty.

\medskip

\subsection{Decomposing regular parabolic and loxodromic isometries}
Given a triple of parameters $\pmb\alpha=(\alpha_1,\alpha_2,\alpha_3)$, a triple
of signs $\pmb\sigma=(\sigma_1,\sigma_2,\sigma_3)$, where $\sigma_i\in\{-1,1\}$ 
and at most one of them is positive, and $\tau\in\mathbb C$, we say that a triple
of points $p_1,p_2,p_3\in\PV\setminus\SV$ is {\it strongly regular\/} with respect
to $\pmb\alpha,\pmb\sigma,\tau$ if: $p_1,p_2,p_3$ are pairwise distinct; $p_2$ is
neither orthogonal to $p_1$ nor to $p_3$; $p_1,p_2,p_3$ do not lie in a same
complex line; $\trace R_{\alpha_3}^{p_3}R_{\alpha_2}^{p_2}R_{\alpha_1}^{p_1}=\tau$;
and the isometry $R_{\alpha_3}^{p_3}R_{\alpha_2}^{p_2}R_{\alpha_1}^{p_1}$ is regular
(Definition~\ref{defi:regular}). We denote by $S_{\pmb\alpha,\pmb\sigma,\tau}$
the space of strongly regular triples with respect to $\pmb\alpha,\pmb\sigma,\tau$.

In the case where $\pmb\alpha:=(\alpha,\alpha,\alpha)$ for a given parameter
$\alpha\in\mathbb S^1\setminus\Omega$, considering the function
$\kappa_{\alpha}:\mathbb C\to\mathbb C$ defined by
$$\kappa_{\alpha}(\tau):=\frac{\tau-3}
{(\alpha^{-2}-\alpha)^3},$$
and denoting
$$t_1:=\tance(p_1,p_2),\quad t_2:=\tance(p_2,p_3),\quad 
t:=\real\frac{\langle p_1,p_2\rangle\langle p_2,p_3\rangle\langle p_3,p_1\rangle
}{\langle p_1,p_1\rangle\langle p_2,p_2\rangle\langle p_3,p_3\rangle},$$
we obtain that $S_{\pmb\alpha,\pmb\sigma,\tau}$
is the real semialgebraic surface in $\mathbb R^3(t_1,t_2,t)$ given
by the equation (see~\cite[Theorem~5.2]{spell})
\begin{equation}
\label{eq:surface}
t_1^2t_2+t_1t_2^2-2t_1t_2t+d_1t^2+d_2t+d_{3}=0,
\end{equation}
and by the inequalities
\begin{equation}
\label{eq:sineq}
\sigma_1\sigma_2t_1>0,\quad \sigma_1\sigma_2t_1>\sigma_1\sigma_2,\quad 
\sigma_2\sigma_3t_2>0,\quad \sigma_2\sigma_3t_2>\sigma_2\sigma_3,\quad 
\sigma_1\sigma_2\sigma_3(2\real\kappa_\alpha(\tau)+1)<0,
\end{equation}
where $d_1:=1+4\chi^2$, 
$d_2:=-4\chi(2\chi\real\kappa_\alpha(\tau)+\imag\kappa_\alpha(\tau))$,
$d_3:=(2\chi\real\kappa_\alpha(\tau)+\imag\kappa_\alpha(\tau))^2$, and
$\chi:=\imag\big(\frac{\alpha}{\alpha^{-2}-\alpha}\big)$.

Note that if $\tau\in\mathbb C$
is such that $2\real\kappa_\alpha(\tau)+1=0$, then $S_{\pmb\alpha,\pmb\sigma,\tau}$
is empty. But $\det[g_{ij}]=\sigma_1\sigma_2\sigma_3(2\real\kappa_\alpha(\tau)+1)$,
where $[g_{ij}]$ is the Gram matrix of the points $p_1,p_2,p_3$. Therefore, if 
$2\real\kappa_\alpha(\tau)+1=0$, a triple of pairwise distinct and pairwise
nonorthogonal points $p_1,p_2,p_3$ satisfying
$\trace R_{\alpha_3}^{p_3}R_{\alpha_2}^{p_2}R_{\alpha_1}^{p_1}=\tau$ is collinear, i.e.,
$p_1,p_2,p_3$ lie in the same complex line. (The space of such triples, with respect
to $\pmb\alpha,\pmb\sigma,\tau$, is also parametrized by \eqref{eq:surface} and satisfy
inequalities obtained by changing the last inequality in \eqref{eq:sineq} by
$\sigma_1\sigma_2\sigma_3(2\real\kappa_\alpha(\tau)+1)=0$.)


\begin{lemma}
\label{lemma:ellinverse}
Given\/ $\alpha\in\mathbb S^1$ and\/ $\tau\in\mathbb C$,
we have\/ $2\real(\kappa_\alpha(\tau))+1=0$ iff\/ $\tau\in\ell_{\alpha^3}$.
\end{lemma}

\begin{proof}
Suppose that $2\real(\kappa_\alpha(\tau))+1=0$. As discussed above,
if $p_1,p_2,p_3\in\PV\setminus\SV$ is a triple of pairwise distinct, 
pairwise nonorthogonal
points such that $\trace R_{\alpha}^{p_3}R_{\alpha_2}^{p_2}R_{\alpha_1}^{p_1}=\tau$, 
then $p_1,p_2,p_3$ lie in a same complex line $L$ and, in this case,
$R_{\alpha_3}^{p_3}R_{\alpha_2}^{p_2}R_{\alpha_1}^{p_1}$ has a fixed point
with eigenvalue $\alpha^3$, namely the polar point of the line $L$. 
By Proposition~\ref{prop:traceeigen}, $\tau\in\ell_{\alpha^3}$. 

On the other hand, given $\tau\in\ell_{\alpha^3}$, we have
$\tau=\tau_{\alpha^2,\alpha}(t)$ (see Subsection~\ref{subsec:tangentlines}), for
some~$t\in\mathbb R$, and
$$2\real\big(\kappa_{\alpha}(\tau_{\alpha^2,\alpha}(t))\big)+1=2\real
\Bigg(\frac{t(1-\alpha^6)}{(\alpha^3-1)^2}
-\frac{\alpha^3}{\alpha^3-1}\Bigg)+1.$$
But, $\real\frac{\alpha^3}{\alpha^3-1}=\frac{1}{2}$ and 
$\real(\frac{1-\alpha^6}{(\alpha^3-1)^2})=0$. In fact, by straightforward
calculations, we have
$$\frac{\alpha^3}{\alpha^3-1}+\frac{\alpha^{-3}}{\alpha^{-3}-1}=1
\quad\text{and}\quad
\frac{1-\alpha^6}{(\alpha^3-1)^2}-\frac{1-\alpha^{-6}}{(\alpha^{-3}-1)^2}=0.$$
Therefore, $2\real\big(\kappa_{\alpha}(\tau_{\alpha^2,\alpha}(t))\big)+1=0$. 
\end{proof}

\begin{prop}
\label{prop:parlox3}
Every regular parabolic or loxodromic isometry in\/ $\PU(2,1)$ admits an\/ 
$\alpha^{(3)}$-decom{\-}position, for any parameter\/ 
$\alpha\in\mathbb S^1\setminus\Omega$.
\end{prop}

\begin{proof}
Let $F\in\SU(2,1)$ be a regular parabolic or loxodromic isometry
and let $\tau:=\trace F$.
If there exists a triple of points $p_1,p_2,p_3$ such that 
$\trace R_\alpha^{p_3}R_\alpha^{p_2}R_\alpha^{p_1}=\tau$,
then $F$ admits an $\alpha^{(3)}$-decomposition since there exists a single
regular parabolic or loxodromic class corresponding to $\tau$ (see 
Section~\ref{sec:conjclasses}). In this way, we prove that the space given, in
$\mathbb R^3(t_1,t_2,t)$, by equation~\eqref{eq:surface} and by inequalities
obtained by substituting the last one in~\eqref{eq:sineq} by 
$\sigma_1\sigma_2\sigma_3(2\real\kappa_\alpha(\tau)+1)\leq0$ is nonempty for some 
choice of signs $\sigma_i$.

For fixed values of $t_1,t_2$ we have a quadratic equation in $t$ (as
$d_1\neq 0$)
with discriminant $(d_2-2t_1t_2)^2-4d_1(t_1^2t_2+t_1t_2^2+d_3)$. Thus,
the equation \eqref{eq:surface} has a solution for the given
values of $t_1,t_2$ iff 
$$d_2^2-4d_2t_1t_2+4t_1^2t_2^2\geq 4d_1(t_1^2t_2+t_1t_2^2+d_3).$$
Since $d_1>0$, this inequality holds if $t_1,t_2\ll 0$. By the same reason,
for a fixed value of $t_1<0$, there exists $t_2>1$ satisfying the inequality above.
In this way, we prove that there are solutions for any choice of signs
satisfying $\sigma_1=-\sigma_2=\sigma_3$ and $-\sigma_1=\sigma_2=\sigma_3$. 
\end{proof}

\begin{rmk}
\label{rmk:tracezero}
From the proof of Proposition~\ref{prop:parlox3}, it follows that if
$F$ is an elliptic isometry with $\trace F=0$, then $F$ admits an 
$\alpha^{(3)}$-decomposition for any parameter $\alpha\in\mathbb S^1\setminus\Omega$.
In fact, it follows from the mentioned proof that there exist $p_1,p_2,p_3\in\PV\setminus\SV$ and $\delta\in\Omega$ such that $F=\delta
R_\alpha^{p_3}R_\alpha^{p_2}R_\alpha^{p_1}$. Since two isometries with
trace $\tau=0$ are conjugated (see Section~\ref{sec:conjclasses}), the result follows.
\end{rmk}

\subsection{Decomposing nonregular isometries}
Since, for any parameter $\alpha\in\mathbb S^1\setminus\Omega$ and any triple
$p_1,p_2,p_3\in\PV\setminus\SV$ of pairwise orthogonal points, we have
$R_\alpha^{p_3}R_\alpha^{p_2}R_\alpha^{p_1}=1$, the identical class admits an
$\alpha^{(3)}$-decomposition. This implies, together with Proposition~\ref{prop:parlox3},
that $\mathsf E_{\alpha^{(3)}}$ contains the nondiagonal side and the vertex
of $\rho(\mathcal E)$. Using the sets
$\mathsf E_{\alpha_1,\alpha_2}^{\sigma_1\sigma_2}$ described in
Proposition~\ref{prop:a1a2eldecomp}, we obtain in the next proposition that
the special elliptic classes in the nondiagonal side of $\rho(\mathcal E)$, i.e.,
those with positive center, also admit an $\alpha^{(3)}$-decomposition.

\begin{prop}
\label{prop:specialdecomp}
Every special elliptic isometry with positive center admits 
an\/ $\alpha^{(3)}$-de{\-}com{\-}po{\-}si{\-}tion, for any parameter\/ $\alpha\in\mathbb S^1
\setminus\Omega$.
\end{prop}

\begin{proof}
Let $\beta\in\mathbb S^1\setminus\Omega$ be a parameter. We will prove that
$\mathsf E_{\alpha,\alpha}^{--}\cup \mathsf E_{\alpha,\alpha}^{++}$ intersects
$\mathsf E_{\beta,\overline\alpha}^{+-}\cup \mathsf E_{\beta,\overline\alpha}^{++}$.
If this intersection occurs in $\mathcal E^{\mathrm{reg}}$, then
there exists a relation of the form 
$R_\alpha^{p_2}R_\alpha^{p_1}=\delta R_{\overline\alpha}^{p_3}R_{\beta}^q$,
where $q\in\EV$ and $\delta\in\Omega$, and, in this case, the result follows.

But, first, we need to consider the case where
$\mathsf E_{\alpha,\alpha}^{--}\cup \mathsf E_{\alpha,\alpha}^{++}$ intersects
$\mathsf E_{\beta,\overline\alpha}^{+-}\cup \mathsf E_{\beta,\overline\alpha}^{++}$
over the boundary of $\rho(\mathcal E)$. In this case, there exists 
$\delta\in\Omega$ such that $\delta\overline\alpha\beta\in\{
\alpha^2,\overline\alpha,-\overline\alpha,\overline\alpha^4\}$. In fact, by 
Proposition~\ref{prop:traceeigen}, the set
$\mathsf E_{\alpha,\alpha}^{--}\cup \mathsf E_{\alpha,\alpha}^{++}$
is contained in the inverse image under $\alttrace$ of the union
$\ell_{\alpha^2}\cup\ell_{\omega\alpha^2}\cup\ell_{\omega^2\alpha^2}$ and, thus,
if it intersects 
$\mathsf E_{\beta,\overline\alpha}^{+-}\cup \mathsf E_{\beta,\overline\alpha}^{++}$ 
over the boundary of $\rho(\mathcal{E})$, it follows that
there exists $\delta\in\Omega$ such that $\ell_{\alpha^2}$ intersects 
$\ell_{\delta\overline{\alpha}\beta}$ over $\partial\Delta$. Moreover,
the points where $\ell_{\alpha^2}$ intersect $\partial\Delta$ are exacty
$\tau_1:=\alpha^2+\overline{\alpha}+\overline{\alpha}$,
$\tau_2:=\alpha^2-\overline{\alpha}-\overline{\alpha}$,
and $\tau_3:=\alpha^2+\alpha^2+\overline\alpha^{4}$. 
So, we can assume
that there exists $\delta\in\Omega$ such that $\beta\in\{\delta\alpha^3,-\delta,\delta
\overline\alpha^{3}\}$.
If $\beta=\delta\alpha^3$ then, for any $p\in\PV\setminus\SV$, 
we have $R_{\alpha}^pR_{\alpha}^p=R_{\alpha^2}^p=
\delta^2 R_{\delta\alpha^3}^pR_{\overline\alpha}^p$, so 
$R_\beta^p$ admits an $\alpha^{(3)}$-decomposition. If $\beta=-\delta$,
taking positive points $p,q\in\EV$ with $\tance(p,q)=1$, the isometry
$R_{\overline\alpha}^qR_\beta^p$ is parabolic isometry with trace 
$-\delta\overline\alpha-\delta\overline\alpha+\delta\alpha^2$ and, by
Proposition~\ref{prop:a1a2loxpardecomp}, admits an $\alpha^{(2)}$-decomposition.
Finally, if $\beta=\delta\overline\alpha^3$, given $p\in\EV$, there exists
$q\in\PV\setminus\SV$ such that $R_{\overline\alpha}^qR_\beta^p$ is a parabolic with trace
$\delta\overline\alpha^4+\delta\alpha^2+\delta\alpha^2$ (see the proof of
Proposition~\ref{prop:a1a2eldecomp}); again by Proposition~\ref{prop:a1a2loxpardecomp},
the isometry $R_{\overline\alpha}^qR_\beta^p$ admits a $\alpha^{(2)}$-decomposition.

\smallskip

Now, if $\beta\notin\{\delta\alpha^3,-\delta,\delta
\overline\alpha^{3}\}$ for every $\delta\in\Omega$, then
$\mathsf E_{\alpha,\alpha}^{++}\cap\mathsf E_{\beta,\overline\alpha}^{+-}=\varnothing$.
In fact, if this intersection is nonempty, since it occurs in $\mathcal E^{\mathrm{reg}}$,
we obtain a relation $R_{\alpha}^{p_2}R_{\alpha}^{p_1}=
\delta_0 R_{\overline\alpha}^{p_3}R_{\beta}^q$, with $\sigma p_1=\sigma p_2=\sigma q=1$, 
$\sigma p_3=-1$, and $\delta_0\in\Omega$. This implies that 
$R_{\alpha}^{p_3}R_{\alpha}^{p_2}=
\delta_0 R_{\beta}^qR_{\overline\alpha}^{p_1}=:R$. Since $\sigma p_2\neq\sigma p_3$,
the isometry $R$ is loxodromic and, by Proposition~\ref{prop:traceeigen}, the lines 
$\ell_{\alpha^2}$ and $\ell_{\delta_0\overline\alpha\beta}$ intersect outside $\Delta$, 
thus $\ell_{\alpha^2}=\ell_{\delta_0\overline\alpha\beta}$ or, equivalently, 
$\alpha^2=\delta_0\overline\alpha\beta$. 
Analogously, $\mathsf E_{\alpha,\alpha}^{--}\cap
\mathsf E_{\beta,\overline\alpha}^{++}=\varnothing$ and 
$\mathsf E_{\alpha,\alpha}^{--}\cap
\mathsf E_{\beta,\overline\alpha}^{-+}=\varnothing$ when 
$\beta\notin\{\delta\alpha^3,-\delta,\delta\overline\alpha^{3}\}$ for every $\delta\in\Omega$.

Suppose that $\beta\notin\{\delta\alpha^3,-\delta,\delta
\overline\alpha^{3}\}$ for every $\delta\in\Omega$. 
We will prove by contradiction that if $\mathsf E_{\alpha,\alpha}^{++}\cap
\mathsf E_{\beta,\overline\alpha}^{++}=\varnothing$, then
$\mathsf E_{\alpha,\alpha}^{--}\cap\mathsf E_{\beta,\overline\alpha}^{+-}
\neq\varnothing$. Suppose that 
$\mathsf E_{\alpha,\alpha}^{--}\cap\mathsf E_{\beta,\overline\alpha}^{+-}=
\varnothing$ and that $\mathsf E_{\alpha,\alpha}^{++}\cap
\mathsf E_{\beta,\overline\alpha}^{++}=\varnothing$. By item~(1) of 
Remark~\ref{rmk:a1a2props}, the subsegment of $\mathsf E_{\beta,\overline\alpha}^{++}$
of slope $-1$ intersects the nondiagonal side of~$\rho(\mathcal E)$ but does not intersect 
the diagonal ($\mathsf E_{\beta,\overline\alpha}^{++}$ has a vertex in 
$\mathcal E^{\text{reg}}$);
denote by $\mathsf C$ the line segment that complements 
$\mathsf E_{\beta,\overline\alpha}^{++}$, i.e., $\mathsf C$ is the segment of 
slope~$-1$ connecting the vertex of $\mathsf E_{\beta,\overline\alpha}^{++}$ lying 
in~$\mathcal E^{\text{reg}}$ and the diagonal side of $\rho(\mathcal E)$.
By item~(2) of Remark~\ref{rmk:a1a2props}, there exists 
$\delta_1\in\Omega$, $\delta_1\neq 1$, such that
$\alttrace(\mathsf C\cap\mathcal E^{\text{reg}})=
\delta_1\alttrace(\mathsf R_{\beta,\overline\alpha}^{+-})$, where 
$\mathsf R_{\beta,\overline\alpha}^{+-}:= \mathsf E_{\beta,\overline\alpha}^{+-}
\cap\mathcal E^{\text{reg}}$. Thus,
$\mathsf E_{\alpha,\alpha}^{++}\cap\mathsf E_{\beta,\overline\alpha}^{++}=
\varnothing$ implies $\mathsf E_{\alpha,\alpha}^{++}\cap\mathsf C\neq\varnothing$,
otherwise we would have two lines tangent to $\partial\Delta$ that do not intersect 
(see Proposition~\ref{prop:traceeigen}). Moreover, as $\mathsf C$ is a segment of
slope~$-1$, such intersection point does not lie in the segment of slope~$-1$
that composes $\mathsf E_{\alpha,\alpha}^{++}$.
Therefore, there exist $p_1,p_2,p_3\in\PV\setminus\SV$,
with $\sigma p_1=\sigma p_2=1$ and $\sigma p_3=-1$, and $\delta_2\in\Omega$, 
such that $\trace R_\alpha^{p_2}R_\alpha^{p_1}=
\delta_2\trace R_{\overline\alpha}^{p_3}R_\beta^q$. Denote
$R:=R_{\alpha}^{p_2}R_{\alpha}^{p_1}$ and $S:=
\delta_2 R_{\overline\alpha}^{p_3}R_\beta^q$.
By item~(1) of Remark~\ref{rmk:a1a2props}, the
line $L_1:=\mathrm L(p_1,p_2)$ is hyperbolic (as the class of 
$R$ does not lie in the segment of slope $-1$ that composes $\mathsf E_{\alpha,\alpha}$). 
The line $L_2:=\mathrm L(p_3,q)$ is also hyperbolic, as $p_3$ and $q$ have opposite
signatures. Let $\tilde p_1,\tilde p_2$ be respectively the points in $L_1$ orthogonal 
to $p_1,p_2$, and let $\tilde p_3,\tilde q$ be respectively the points in $L_2$
orthogonal to $p_3,q$. Define $\widetilde R:=R_\alpha^{\tilde p_2}R_\alpha^{\tilde p_1}$
and $\widetilde S:=\delta_2 R_{\overline\alpha}^{\tilde p_3}R_\beta^{\tilde q}$.
By Remark~\ref{rmk:rtilde},
$R$ and $\widetilde R$
are isometries with the same trace but lying in distinct $\SU(2,1)$-conjugacy
classes, and the same holds for the isometries $S$ and $\widetilde S$.
Since trace determines eigenvalues in $\SU(2,1)$, the eigenvalues of the 
isometries $R,\widetilde R,S,\widetilde S$ are 
$\alpha^2,\delta_2\overline\alpha\beta,\delta_2^2\overline\alpha\overline\beta$.
In this way, one of the isometries $R,\widetilde R$ lies in the same $\SU(2,1)$-conjugacy
class of one of the isometries $S,\widetilde S$.
Moreover, as the lines $L_1,L_2$ are hyperbolic, the negative type eigenvalue of both 
$R$ and $\widetilde R$ is not $\alpha^2$, and the negative type eigenvalue of
both $S$ and $\widetilde S$ is not $\delta_2\overline\alpha\overline\beta$.
Suppose that the negative type eigenvalue of $S$ is 
$\delta_2^2\overline\alpha\overline\beta$. Then either $[R]=[S]$
(which contradicts $\mathsf E_{\alpha,\alpha}^{--}\cap
\mathsf E_{\beta,\overline\alpha}^{+-}=\varnothing$) or $[\widetilde R]=[S]$ (which
contradicts $\mathsf E_{\alpha,\alpha}^{++}\cap\mathsf
E_{\beta,\overline\alpha}^{+-}=\varnothing$). Therefore,
the negative eigenvalue of $S$ is $\alpha^2$. This implies that either
$[R]=[\widetilde S]$ (which contradicts $\mathsf E_{\alpha,\alpha}^{--}\cap
\mathsf E_{\beta,\overline\alpha}^{-+}=\varnothing$) or $[\widetilde R]=[\widetilde S]$
(which contradicts $\mathsf E_{\alpha,\alpha}^{++}\cap
\mathsf E_{\beta,\overline\alpha}^{-+}=\varnothing$).
\end{proof}


\begin{prop}
\label{prop:2stepdecomp}
If\/ $\alpha^3\in\Omega\cup-\Omega$, every\/ $2$-step unipotent isometry admits 
an\/ $\alpha^{(3)}$-decomposition.
\end{prop}

\begin{proof}
Suppose that $\alpha^3\in\Omega$, i.e., $\alpha^3=\omega^j$ for some $j=1,2$.
Let $U\in\SU(2,1)$ be a $2$-step unipotent isometry fixing an isotropic point
$v\in\SV$ with eigenvalue $1$.
If $p\in\mathbb Pv^\perp$ is a nonisotropic point, then the isometry 
$R:=R_{\overline\alpha}^{p}U$ fixes $v$ with eigenvalue $\overline\alpha$
and fixes $p$ with eigenvalue $\alpha^2$. Since $\alpha^3\neq 1$, $R$ is
ellipto-parabolic with $\trace R=2\overline\alpha+\alpha^2=\tau_{\alpha,\alpha}(0)$.
But $\alpha^3=\omega^j$ implies
$\tau_{\alpha,\alpha}(0)=\omega^j\tau_{\alpha,\alpha}(1)$. Therefore,
by Proposition~\ref{lemma:euclideanline}, 
the isometry $R$ admits an $\alpha^{(2)}$-decomposition, implying that
$U$ admits an $\alpha^{(3)}$-decomposition.

If $\alpha^3\in-\Omega$ we are in the case of involutions and the proposition 
follows directly from~\cite[Proposition~16]{Will2017}.
\end{proof}

\begin{rmk}
\label{rmk:2stepdecomp}
It remains an open problem to determine if a $2$-step unipotent isometry admits an
$\alpha^{(3)}$-decomposition for parameters satisfying 
$\alpha^3\notin\Omega\cup-\Omega$. 
\end{rmk}

We now determine which regular elliptic isometries and special elliptic isometries
with negative center admit an $\alpha^{(3)}$-decomposition.
In order to do so, it suffices to describe $\mathsf E_{\alpha^{(3)}}$, since its intersection
with the diagonal side of $\rho(\mathcal E)$ corresponds uniquely to the classes of 
special elliptic isometries with negative center. Such description is based on the 
{\it product map\/}.

\subsection{The product map}
\label{subsec:productmap}
In this subsection we briefly summarize some definitions and results that can be
found in \cite{FW2009,Paupert2007,Will2017}.
Consider the spaces $\mathcal G, c(\mathcal G)$ and the projection 
$\rho:\mathcal C\to c(\mathcal G)$ as defined in Subsection~\ref{subsec:puclasses}.
Given two semisimple $\PU(2,1)$-conjugacy classes $C_1,C_2\in c(\mathcal G)$ the
{\it product map\/}, with respect to the given classes,
is the function $\widetilde\mu:C_1\times C_2\to\mathcal G$ defined by
$\widetilde\mu(A,B)=[AB]$, where $[I]$ denotes the conjugacy class of
the isometry $I$. In what follows, we will mainly consider the function 
(which we also 
refer as product map) $\overline\mu:C_1\times C_2\to c(\mathcal G)$, defined by
$\overline\mu:=\rho\circ\widetilde\mu$ .

\begin{defi}
\label{defi:reducible}
We say that a subgroup $\Gamma$ of $\PU(2,1)$ is {\it reducible}\/ if it fixes
a point in $\PV$. Given isometries $A,B\in\PU(2,1)$, we say that the pair $(A,B)$ 
is reducible if it generates a reducible group.
If a subgroup is not reducible we say that it is {\it irreducible}.
Given two semisimple conjugacy classes $C_1$ and $C_2$, the image under 
$\overline\mu:C_1\times C_2\to c(\mathcal G)$ of reducible pairs is called
{\it reducible walls\/}.
\end{defi}

In terms of the above definition, we have the following
properties of the product map.

\begin{prop}
\label{prop:proptildemu}
Given two semisimple conjugacy classes\/ $C_1$ and\/ $C_2$ we have:

\smallskip

$\bullet$ $\overline\mu$ is proper; in particular the image\/ 
$\overline\mu(C_1\times C_2)$ is closed in\/ $c(\mathcal  G)$;

$\bullet$ the image of an irreducible pair in\/ $C_1\times C_2$ under\/ 
$\overline\mu$ is an interior point of\/ $\overline\mu(C_1\times C_2)$;

$\bullet$ the reducible walls, divide $c(\mathcal G)$ in closed {\it chambers}.
Each of these chambers is either full or empty;

$\bullet$ the intersection of the reducible walls of 
$\overline\mu(C_1\times C_2)$ with $\rho(\mathcal E)$ is given by the union of finitely many
line segments of slopes $-1,\frac 12,2$.
\end{prop}

\subsection{Dividing $\mathsf E_{\alpha^{(3)}}$ into chambers}
\label{subsec:fullempty}
We prove (see Corollary~\ref{cor:cuptildemu}) that $\mathsf E_{\alpha^{(3)}}$ is the
intersection of $\rho(\mathcal E)$ with the union of all images 
$\overline\mu(C_1\times C_2)$, where $C_1$ is the class of a special elliptic isometry 
with parameter~$\alpha$, and $C_2$ is a semisimple class admitting an 
$\alpha^{(2)}$-decomposition. It is quite hard to directly obtain such union, writing 
down every image (as done in~\cite{Will2017} for the case of involutions). 
So, we do this indirectly using {\it bendings\/} (see~\cite{spell}).

\smallskip

Given a product $R:=R_{\alpha_2}^{p_2}R_{\alpha_1}^{p_1}$, where $p_1,p_2$ are
nonisotropic points and $\alpha_1,\alpha_2\in\mathbb S^1\setminus\Omega$
are parameters, if $C\in\SU(2,1)$ is an isometry in the centralizer of
$R$, we have $R_{\alpha_2}^{p_2}R_{\alpha_1}^{p_1}=
R_{\alpha_2}^{Cp_2}R_{\alpha_1}^{Cp_1}$. These are called {\it bending relations\/}. 
Moreover, by \cite[Proposition~4.3]{spell}, there exists a one-parameter
subgroup $B:\mathbb R\to\SU(2,1)$ such that $B(s)$ is in the centralizer of $R$
for every $s\in\mathbb R$ and for every isometry $C$ that commutes with
$R$, there exists $s\in\mathbb R$ with $Cp_i=B(s)p_i$, $i=1,2$. We say
that $B(s)$ is a {\it bending\/} of $R$. 
Bendings act on $p_1,p_2$ by moving these points
over metric circles, hypercycles, horocycles, contained in the line
$\mathrm L(p_1,p_2)$, depending on the nature of the isometry~$R$.

\begin{rmk}
\label{rmk:nondeggeneral}
Let $\alpha\in\mathbb S^1\setminus\Omega$ be a parameter and let 
$p_1,p_2,p_3\in\PV\setminus\SV$. Consider the isometry $F:=R_\alpha^{p_3}
R_\alpha^{p_2}R_\alpha^{p_1}$.
If $p_1$ and $p_2$ are distinct nonorthogonal points and 
$L:=\mathrm L(p_1,p_2)$ is noneuclidean, then a nontrivial 
bending $B(s)$ of $R_1$ satisfies~$\tance(B(s)p_2,p_3)=\tance(p_2,p_3)$ iff
either ($p_3$ is orthogonal to a fixed 
point of $R_1$ and $p_3\not\in\mathrm L(p_1,p_2)$) or ($p_3\in\mathrm L(p_1,p_2)$ and 
$p_3$ is $R_1$-fixed). If $L$ is Euclidean, then $\tance(B(s)p_2,p_3)=\tance(p_2,p_3)$
iff either ($p_3\not\in L$ and $p_3$ is orthogonal to an $R_1$-fixed point) or 
$p_3\in L$.
\end{rmk}

In what follows, for an isometry $I\in\SU(2,1)$, $[I]$ denotes the $\PU(2,1)$-conjugacy
class of the corresponding isometry in~$\PU(2,1)$.

\begin{lemma}
\label{lemma:r1parabolic}
Let\/ $F\in\SU(2,1)$ be an elliptic isometry admitting a decomposition of the form\/
$F=\delta R_\alpha^{p_3}R_\alpha^{p_2}R_\alpha^{p_1}$, where 
$p_1,p_2,p_3\in\PV\setminus\SV$ and $\delta\in\Omega$, such
that\/ $R_\alpha^{p_2}R_\alpha^{p_1}$ is parabolic. If\/~$[F]\neq
[R_{\overline\alpha^3}^q]$, $q\in\EV$, then there exist\/ 
$q_1,q_2,q_3\in\PV\setminus\SV$ such that\/
$F=\delta R_\alpha^{q_3}R_\alpha^{q_2}R_\alpha^{q_1}$ and\/ $R_\alpha^{q_2}R_\alpha^{q_1}$
is either regular elliptic or loxodromic.  
\end{lemma}

\begin{proof}
Denote $R_1:=R_\alpha^{p_2}R_\alpha^{p_1}$, $R_2=R_\alpha^{p_3}R_\alpha^{p_2}$,
$R_3=R_\alpha^{p_1}R_\alpha^{p_3}$,
$L_1:=\mathrm L(p_1,p_2)$, $L_2:=\mathrm L(p_2,p_3)$, and $L_3:=\mathrm L(p_1,p_3)$.
If $R_2$ or $R_3$ is regular elliptic or loxodromic, since $[F]=[R_\alpha^{p_1}
R_\alpha^{p_3}R_\alpha^{p_2}]=[R_\alpha^{p_2}R_\alpha^{p_1}R_\alpha^{p_3}]$, 
the result follows. So, we can assume that the points $p_1,p_2,p_3$ have the 
same signature. We can also assume that $p_1,p_2,p_3$ do not lie a same Euclidean line,
otherwise the isometry~$F$ is not elliptic.

The idea now is, by bending $R_2$, to obtain new centers $B(s)p_2,B(s)p_3$ such that
$\tance(p_1,B(s)p_2)\neq\tance(p_1,p_2)$, which implies that the isometry 
$R_\alpha^{B(s)p_2}R_\alpha^{p_1}$ is 
either regular elliptic or loxodromic (see \cite[Corollary~5.10]{spell}).
Note that this approach does not work if $\langle p_2,p_3\rangle=0$ or
if $p_1$ is orthogonal to a fixed point of $R_2$.

The fact that $R_1$ is parabolic implies that $p_1,p_2$ are distinct nonorthogonal
points of same signature. Hence, if $p_3$ is in $L_1$, $p_3$ cannot be an 
$R_1$-fixed point (otherwise $L_1=L_2$ is Euclidean) and thus,
by bending $R_1$ if necessary, we can assume
that $p_2,p_3$ are also distinct and nonorthogonal. So, by bending $R_2$, 
the result follows (see Remark~\ref{rmk:nondeggeneral}).

Suppose that $\langle p_2,p_3\rangle=0$. By what was discussed above,
we can assume that $L_1$ is Euclidean and $L_2$ is spherical. In this case,
$p_1$ is not orthogonal to any fixed point of $R_2$; hence, there is a bending 
$B(s)$ of $R_2$ such that $B(s)p_2$ and $p_3$ are nonorthogonal points
and the result follows.

It remains to consider the case where $p_i\in\EV$,
$L_i$ is Euclidean, and $p_i$ is orthogonal to a fixed point of $R_{i+1}$ 
(indices modulo $3$). By Lemma~\ref{lemma:euclideanline}, for each $i$, there exists 
$c_i\in L_i$ such that $c_i$ is a fixed point of $R_i$ with eigenvalue 
$\delta\alpha^{-3}$. Therefore, $F=\delta R_{\overline\alpha^3}^{q}$ for some 
$q\in\PV\setminus\SV$. If $q\in\BV$, applying a simultaneous change of signs in 
$R_{\overline\alpha}^{p_3}R_{\overline\alpha^{3}}^{q}$ (see Remark~\ref{rmk:rtilde}) we obtain a relation of the form
$R_\alpha^{q_3}R_\alpha^{p_2}R_\alpha^{p_1}=\delta 
R_{\overline\alpha^3}^{\tilde q}$, where $q_3\in\BV$ and $\tilde q\in\EV$, in which
case $R_\alpha^{q_3}R_\alpha^{p_2}$ is loxodromic and
the result follows as above.
\end{proof}

\begin{cor}
\label{cor:cuptildemu}
For any parameter\/ $\alpha\in\mathbb S^1\setminus\Omega$, we have 
\begin{equation}
\label{eq:ea3cup}
\mathsf E_{\alpha^{(3)}}=\bigcup\limits_{
C_1\in\mathsf S_{\alpha},\, 
C_2\in\mathsf G_{\alpha^{(2)}}}
\overline\mu(C_1\times C_2)\cap\rho(\mathcal E).
\end{equation}
\end{cor}

\begin{proof}
By its definition (see Notation~\ref{notation:egclasses}), 
$\mathsf E_{\alpha^{(3)}}$ contains the right side of~\eqref{eq:ea3cup}.

Conversely, by Propositions~\ref{prop:parlox3} and~\ref{prop:specialdecomp}, we need to
prove that given a regular elliptic isometry or a special elliptic isometry 
with negative center $F\in\SU(2,1)$ admitting a decomposition of the form 
$F=\delta R_{\alpha}^{p_3}R_{\alpha}^{p_2}R_{\alpha}^{p_1}$, 
where $R_1:=R_\alpha^{p_2}R_\alpha^{p_1}$ is parabolic and $\delta\in\Omega$, then 
the $\PU(2,1)$-conjugacy class of $F$ lies in an image 
$\overline\mu(C_1\times C_2)$, for some $C_1\in\mathsf S_{\alpha}$ and 
$C_2\in \mathsf G_{\alpha^{(2)}}$. This follows 
directly from Lemma~\ref{lemma:r1parabolic}.
\end{proof}

\begin{lemma}
\label{lemma:reduciblepair}
Let\/ $\alpha\in\mathbb S^1\setminus\Omega$ be a parameter and
consider a reducible pair\/ $(R_{\alpha}^{p_3},R_{\alpha}^{p_2}R_{\alpha}^{p_1})$
such that\/ $R_{\alpha}^{p_3}R_{\alpha}^{p_2}R_{\alpha}^{p_1}$ is regular
elliptic and the points\/ $p_1,p_2,p_3\in\PV\setminus\SV$ do not lie in a same
complex line. Then the image under\/ $\overline\mu$ of such pair is an interior
point of\/ $\mathsf E_{\alpha^{(3)}}$.
\end{lemma}

\begin{proof}
Define $L_1:=\mathrm L(p_1,p_2)$, $R_1:=R_{\alpha}^{p_2}R_{\alpha}^{p_1}$,
$R_2:=R_\alpha^{p_3}R_\alpha^{p_2}$, and $F:=R_{\alpha}^{p_3}R_{\alpha}^{p_2}R_{\alpha}^{p_1}$.
Suppose that the pair $(R_{\alpha}^{p_3},R_1)$ is reducible and that
$p_3\notin L_1$. Then $p_3$ is orthogonal to a fixed point of $R_1$.

Note that 
$R_\alpha^{p_1}FR_{\overline\alpha}^{p_1}=R_\alpha^{p_1}R_\alpha^{p_3}R_\alpha^{p_2}$. Then
if $p_1$ is not orthogonal to a fixed point of~$R_2$, the image under $\overline\mu$ of
the irreducible pair $(R_\alpha^{p_1},R_2)$, which by Proposition~\ref{prop:proptildemu}
is an interior point in the image
of~$\overline\mu$, coincides with the image (under a distinct $\overline\mu$) of the pair
$(R_\alpha^{p_3},R_1)$. Hence, we can also assume that $p_1$ is orthogonal to a fixed
point of $R_2$.

Suppose that $p_3$ is not orthogonal to $p_2$.
We will prove that there exists a triple $q_1,q_2,q_3\in\PV\setminus\SV$ of points
not lying in a same complex line such that
$F=R_\alpha^{q_3}R_\alpha^{q_2}R_\alpha^{q_1}$ and~$q_3$ is not orthogonal
to a fixed point of $R_1$, i.e., $(R_\alpha^{q_3},R_\alpha^{q_2}R_\alpha^{q_1})$ is
an irreducible pair. Note that if at most one of the points $p_i$ is positive, then the triple
$p_1,p_2,p_3$ is strongly regular with respect to $\pmb\alpha:=(\alpha,\alpha,\alpha)$,
$\pmb\sigma:=(\sigma p_1,\sigma p_2,\sigma p_3)$, and $\tau:=\trace F$, so the 
result follows from~\cite[Lemma~5.5]{spell}. Thus, we can assume that at least
two of the points $p_1,p_2,p_3$ is positive. Suppose that one of the isometries $R_1,R_2$
is loxodromic. In this case, using simultaneous change of signs (together, if necessary, with
a conjugation that cyclic permutes the points $p_1,p_2,p_3$ as above)
we obtain a strongly regular triple $q_1,q_2,q_3$ with 
$F=R_\alpha^{q_3}R_\alpha^{q_2}R_\alpha^{q_1}$ and the result follows as in the previous case.
Finally, we assume that the points $p_1,p_2,p_3$ are positive.
Note that neither~$R_1$ nor~$R_2$ can be parabolic, since~$F$ is regular elliptic. Then~$R_1$ 
and~$R_2$ are regular elliptic and, if~$a$ is the $R_1$-fixed point
orthogonal to~$p_3$ and~$b$ is the $R_2$-fixed point orthogonal to~$p_1$,
then~$a$ and~$b$ are nonisotropic points with $\langle a,b\rangle=0$. Denote
$L:=\mathrm L(p_1,a)$; then $L=\mathbb Pb^\perp$ and $L=L_1$. It follows
that $p_2$ is orthogonal to $b$ and either $p_3\in L_1$ or $L_1=\mathbb Pp_3^\perp$,
both cases contradicting the hypothesis.

Now, suppose that $\langle p_3,p_2\rangle=0$ and $\langle p_3,p_1\rangle\neq 0$. 
If follows that $p_2$ is a fixed point of $R_2$ which implies that the triple
$p_1,p_2,p_3$ is pairwise orthogonal and the isometry $F$ is not regular elliptic,
a contradiction. (We obtain the same contradiction supposing that 
$\langle p_3,p_1\rangle=0$ and $\langle p_3,p_2\rangle\neq 0$.)

Finally, assume $\langle p_3,p_2\rangle=\langle p_3,p_1\rangle=0$, i.e.,
$p_3$ is the polar point of the line $L_1$. By
Proposition~\cite[Corollary~3.7]{spell}, since $F$ is regular elliptic, the points
$p_1,p_2,p_3$ cannot be pairwise orthogonal. 
In this case, we cannot bend the decomposition of $F$. But, if $q$ is any 
point in the line $\mathrm L(p_2,p_3)$ with
$\sigma q_2=\sigma p_2$, and $\tilde q$ is the point in $\mathrm L(p_2,p_3)$
orthogonal to $q_2$, we have
$R_{\alpha}^{\tilde q}R_{\alpha}^{q}=R_\alpha^{p_3}R_{\alpha}^{p_2}=
R_{\overline\alpha}^c$,
where $c$ is the polar point of the line $\mathrm L(p_2,p_3)$.
So, we can without loss of generality assume that
$\langle \tilde q, p_1\rangle\neq 0$, and proceed as in the previous paragraph.
\end{proof}

\begin{prop}
\label{prop:a3walls}
Let\/ $\alpha\in\mathbb S^1\setminus\Omega$ be a parameter. The set\/
$\mathsf E_{\alpha^{(3)}}$ is given by the union of closed chambers in\/ 
$\rho(\mathcal E)$ delimited by\/ 
$\alttrace^{-1}(\ell_{\alpha^3}\cup\ell_{\omega\alpha^3}\cup
\ell_{\omega^2\alpha^3})\subset\mathsf T$. Each of these chambers is either full 
or empty.
\end{prop}

\begin{proof}
We will show that every convergent sequence $x_n\in\rho(\mathcal E)$ such that,
for all~$n$,
$x_n$ lies in a reducible wall of $\overline\mu(C_{1,n}\times C_{2,n})$, 
for some sequences of classes $C_{1,n}\in\mathsf S_{\alpha}$ and 
$C_{2,n}\in\mathsf G_{\alpha^{(2)}}$,
converges either to an interior point of $\mathsf E_{\alpha^{(3)}}$ or to
a point lying in a reducible wall of $\overline\mu(C_1\times C_2)$, for some
semisimple classes $C_1\in\mathsf S_{\alpha}$ and 
$C_2\in \mathsf G_{\alpha^{(2)}}$. Hence, the first part of the 
result follows from Lemma~\ref{lemma:reduciblepair}.

By Corollary~\ref{cor:cuptildemu} (and remembering that $\mathsf S_{\alpha}$ consists of only 
two elements),
by considering a subsequence if necessary, 
we can assume that there exists $p_3\in\PV\setminus\SV$ such that 
$C_{1,n}=[R_\alpha^{p_3}]=:C_1$, for all $n$.

We can also assume that $\rho(C_{2,n})\in\rho(\mathcal E)$,
for all $n$. In fact, let $F\in\SU(2,1)$ be a regular elliptic isometry that admits a
decomposition $F=R_\alpha^{p_3}R_\alpha^{p_2}R_\alpha^{p_1}$, for points
$p_1,p_2,p_3\in\PV\setminus\SV$, such that $A:=R_\alpha^{p_2}R_\alpha^{p_1}$
is loxodromic and the pair $(R_\alpha^{p_3},A)$ is reducible.
Then either $p_3\in\mathrm L(p_1,p_2)$ or $p_3$ is orthogonal to an isotropic fixed
point $v$ of $A$. 
But, if $\langle p_3,v\rangle=0$, then $R_{\alpha}^{p_3}$ also fixes $v$
and, therefore, $F$ is not regular elliptic. Hence, 
$p_3\in\mathrm L(p_1,p_2)$ and, by Lemma~\ref{lemma:ellinverse},
$\alttrace F\in\ell_{\delta\alpha^3}$ for some $\delta\in\Omega$.

So, we have a convergent sequence of points $x_n\in\mathcal E^{\text{reg}}
\subset\rho(\mathcal E)$, each one 
lying in a reducible wall of $\overline\mu(C_1\times C_{2,n})$, where 
$C_{2,n}\in\mathsf G_{\alpha^{(2)}}$ is such that $\rho(C_{2,n})\in\rho(\mathcal E)$. 
Since $\rho(\mathcal E)$ is compact in $c(\mathcal G)$, the sequence $x_n$ converges
to a point $x\in\rho(\mathcal E)$. As 
$\mathcal E\cup\mathcal B$ is 
compact in $\mathcal G$, the sequence $C_{2,n}$ has a subsequence converging to
a class $C_2\in\mathcal E\cup\mathcal B$. If $C_2\notin\mathsf G_{\alpha^{(2)}}$ (note that
$\mathsf G_{\alpha^{(2)}}$ is not closed in $\mathcal G$, see for instance
Proposition~\ref{prop:a1a2eldecomp}), then $C_2$
corresponds to a parabolic class admitting a $\alpha^{(2)}$-decomposition. In this case, 
by Lemma~\ref{lemma:r1parabolic}, every point in $\overline\mu(C_1\times C_2)$ is either
interior in $\mathsf E_{\alpha^{(3)}}$ or lie in $\alttrace^{-1}(\ell_{\alpha^3}\cup\ell_{\omega\alpha^3}\cup\ell_{\omega^2\alpha^3})\subset\rho(\mathcal E)$.
If $C_2\in\mathsf G_{\alpha^{(2)}}$, then $x$ lies in a reducible wall 
of $\overline\mu(C_1\times C_2)$.

The second part follows from Proposition~\ref{prop:proptildemu}.
\end{proof}

By the above proposition, if we consider $\rho(\mathcal E)$ divided into chambers
by the set given by $\alttrace^{-1}(\ell_{\alpha^3}\cup\ell_{\omega\alpha^3}\cup
\ell_{\omega^2\alpha^3})$, i.e., $\rho(\mathcal E)$ divided by the segments
given in Lemma~\ref{lemma:walls},
we obtain the region $\mathsf E_{\alpha^{(3)}}$ by finding which of this
chambers are full/empty. Figure~\ref{fig:alttracetang} gives us an idea
of how these chambers may look like.

\subsection{Deciding which chambers are full}
\label{subsec:fullchambers}
Here we present tools to determine which of the chambers of 
$\mathsf E_{\alpha^{(3)}}$ are full/empty, for a given parameter
$\alpha\in\mathbb S^1\setminus\Omega$.

\begin{prop}
\label{prop:nondiag}
Let\/ $\alpha\in\mathbb S^1\setminus\Omega$ be a parameter and consider\/ 
$\mathsf E_{\alpha^{(3)}}$ decomposed in the union of chambers defined by\/ 
$\alttrace^{-1}(\ell_{\alpha^3}\cup\ell_{\omega\alpha^3}\cup
\ell_{\omega^2\alpha^3})\subset\mathsf T$.
Then, every chamber that contain an open segment of the nondiagonal side of\/ 
$\rho(\mathcal E)$ in its closure is full.
\end{prop}

\begin{proof}
As the nondiagonal side of $\rho(\mathcal E)$ corresponds to the classes of
special elliptic isometries with positive center and to the classes of
ellipto-parabolic isometries, the result follows from 
Propositions~\ref{prop:parlox3}, \ref{prop:specialdecomp} 
and~\ref{prop:a3walls}.
\end{proof}

\begin{prop}
\label{prop:diagchamb}
Let\/ $0<\theta<2\pi$ be such that\/ $(\theta,\theta)$ does not lie in\/
$\alttrace^{-1}
(\ell_{\alpha^3}\cup\ell_{\omega\alpha^3}\cup\ell_{\omega^2\alpha^3})$,
and let\/ $\beta:=e^{\frac{\theta}{3}i}$. 
Then the chamber of\/ $\mathsf E_{\alpha^{(3)}}$ containing\/ 
$(\theta,\theta)$ in its closure is full 
iff\/ $\mathsf E_{\alpha,\alpha}^{--}\cup\mathsf E_{\alpha,\alpha}^{++}$
intersects\/ $\mathsf E_{\beta,\overline\alpha}^{--}\cup
\mathsf E_{\beta,\overline\alpha}^{-+}$.
\end{prop}

\begin{proof}
If $\mathsf E_{\alpha,\alpha}^{--}\cup\mathsf E_{\alpha,\alpha}^{++}$
intersects\/ $\mathsf E_{\beta,\overline\alpha}^{--}\cup
\mathsf E_{\beta,\overline\alpha}^{-+}$, arguing as in
the proof of Proposition~\ref{prop:specialdecomp}, we obtain that there exists a relation
of the form $R_\alpha^{p_2}R_\alpha^{p_1}=
\delta R_{\overline\alpha}^{p_4}R_{\beta}^{p_3}$
where $\delta\in\Omega$ and $p_3\in\BV$. 
Thus, the isometry $R_\beta^{p_3}$, with angle pair $(\theta,\theta)$,
admits a $\alpha^{(3)}$-decomposition and, since 
$(\theta,\theta)$ does not lie in\/
$\alttrace^{-1}
(\ell_{\alpha^3}\cup\ell_{\omega\alpha^3}\cup\ell_{\omega^2\alpha^3})$,
by Proposition~\ref{prop:nondiag},
the chamber containing $(\theta,\theta)$ in its closure is full. 

Conversely, if the chamber containing $(\theta,\theta)$ in its closure is
full, then any isometry with angle pair $(\theta,\theta)$ admits an 
$\alpha^{(3)}$-decomposition which implies that 
$\mathsf E_{\alpha,\alpha}^{--}\cup\mathsf E_{\alpha,\alpha}^{++}$ intersects
$\mathsf E_{\beta,\overline\alpha}^{--}\cup\mathsf E_{\beta,\overline\alpha}^{-+}$.
\end{proof}

\begin{figure}[!ht]
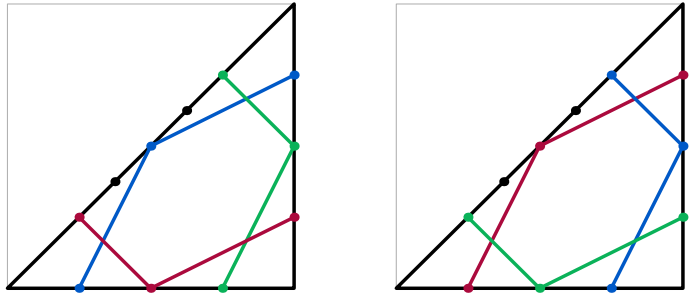

\centering
\includegraphics{pics/abexamp00.mps}
\hspace{1cm}
\includegraphics{pics/abexamp01.mps}
\caption{Walls of the $\alpha^{(3)}$-decomposition, for $\alpha=e^{\frac{\pi}9i}$
(left) and $\alpha=e^{\frac{\pi}3i}$ (right).
}
\label{fig:wallsa3}
\end{figure}

\begin{example}
We apply the previous propositions to obtain the polygonal region
$\mathsf E_{\alpha^{(3)}}$ in the case $\alpha=e^{ai}$, $a=\pi/9$. By
Lemma~\ref{lemma:walls}, $\alttrace^{-1}(\ell_{\alpha^3})$
is given by the two segments connecting the points
$(\pi/2,0),(\pi,\pi),(2\pi,3\pi/2)$; $\alttrace^{-1}(\ell_{\omega\alpha^3})$
is given by the segments connecting $(3\pi/2,3\pi/2),(2\pi,\pi),(3\pi/2,0)$;
and $\alttrace^{-1}(\ell_{\omega^2\alpha^3})$ is given
by the segments connecting $(2\pi,\pi/2)$, $(\pi,0)$, $(\pi/2,\pi/2)$.
These sets are represented in the Figure~\ref{fig:wallsa3} in blue, green, and red,
respectively (see online version).
The central hexagonal chamber is full since it contains the angle pair 
$(\pi/3,\pi/3)$ corresponding to the class where the trace vanishes 
(see Remark~\ref{rmk:tracezero}). Moreover,
Proposition~\ref{prop:nondiag} implies that all but two
chambers are full; the remaining chambers we must check to be 
full or not are those that contain an open segment of the diagonal but don't intersect 
nondiagonal sides.

We take two points in the diagonal, one in each of these regions and not
lying one of the walls, and apply Proposition~\ref{prop:diagchamb}. 
The points $(3\pi/4,3\pi/4)$ and $(5\pi/4,5\pi/4)$ satisfy this condition, and
are marked in Figure~\ref{fig:wallsa3}.

\begin{figure}[h!]
\centering
\includegraphics{pics/abexamp02.mps}
\hspace{1cm}
\includegraphics{pics/abexamp03.mps}
\caption{The sets $\mathsf E_{\beta,\overline\alpha}^{--}\cup
\mathsf E_{\beta,\overline\alpha}^{-+}$ and 
$\mathsf E_{\alpha,\alpha}^{--}\cup
\mathsf E_{\alpha,\alpha}^{++}$ for $\alpha=e^{\frac{\pi}9i}$, 
$\beta=e^{\frac{\pi}{4}i}$ (left) and $\beta=e^{\frac{5\pi}{12}i}$ (right)
}
\label{fig:examppi9}
\end{figure}

So, for $\beta:=e^{\frac{\theta}3i}$, where $\theta=3\pi/4$
or $\theta=5\pi/4$, we need to verify whether the sets 
$\mathsf E_{\beta,\overline\alpha}^{--}\cup\mathsf E_{\beta,\overline\alpha}^{-+}$
and 
$\mathsf E_{\alpha,\alpha}^{--}\cup\mathsf E_{\alpha,\alpha}^{++}$ intersect or
not.
In Figure~\ref{fig:examppi9} we picture these sets;
$\mathsf E_{\beta,\overline\alpha}^{--}\cup
\mathsf E_{\beta,\overline\alpha}^{-+}$ is in light blue/orange,
and the $\mathsf E_{\alpha,\alpha}^{--}\cup\mathsf E_{\alpha,\alpha}^{++}$
is in dark blue/dark green (following the color scheme for each
pair of signs as in Figure~\ref{fig:abdecomptr}; see online version). We see that in both 
cases these sets intersect and both chambers are full.

\begin{figure}[!ht]
\centering
\includegraphics{pics/abexamp04.mps}
\hspace{1cm}
\includegraphics{pics/abexamp05.mps}
\caption{The sets $\mathsf E_{\beta,\overline\alpha}^{--}\cup
\mathsf E_{\beta,\overline\alpha}^{-+}$ and 
$\mathsf E_{\alpha,\alpha}^{--}\cup
\mathsf E_{\alpha,\alpha}^{++}$ for $\alpha=e^{\frac{\pi}3i}$,
$\beta=e^{\frac{\pi}{4}i}$ (left) and
$\beta=e^{\frac{5\pi}{12}i}$ (right)
}
\label{fig:examppi9-2}
\end{figure}

\smallskip

In the case where $a=\pi/3$, the walls are given by the same set of segments as
in the previous case (see Figure~\ref{fig:wallsa3}) but with a cyclic permutation 
of colors. Note that this case is the one of involutions and Figure~\ref{fig:wallsa3} 
and the corresponding figure in~\ref{subsec:a3examples}
are~\cite[Figure~11]{Will2017}. 
Taking the same points in the diagonal side of $\mathcal E$, the sets 
$\mathsf E_{\beta,\overline\alpha}^{--}\cup\mathsf E_{\beta,\overline\alpha}^{-+}$ 
are now given by Figure~\ref{fig:examppi9-2} and we obtain that both corresponding 
chambers are empty.

\end{example}

\subsection{Elliptic isometries that admit an $\alpha^{(3)}$-decomposition}
\label{subsec:a3examples}
Proceeding as in the example above, we determine which chambers are full/empty
for the cases where $a$ is a integral multiple of $\pi/27$, $0<a<2\pi/3$.
(This choice will be clear in the proof of the main theorem.)
The result is given by the following pictures. For each value of $\alpha$,
the chambers in gray are full and the chambers in white are empty.
The union of the chambers in gray is the region $\mathsf E_{\alpha^{(3)}}$ for the given
value of $\alpha$.

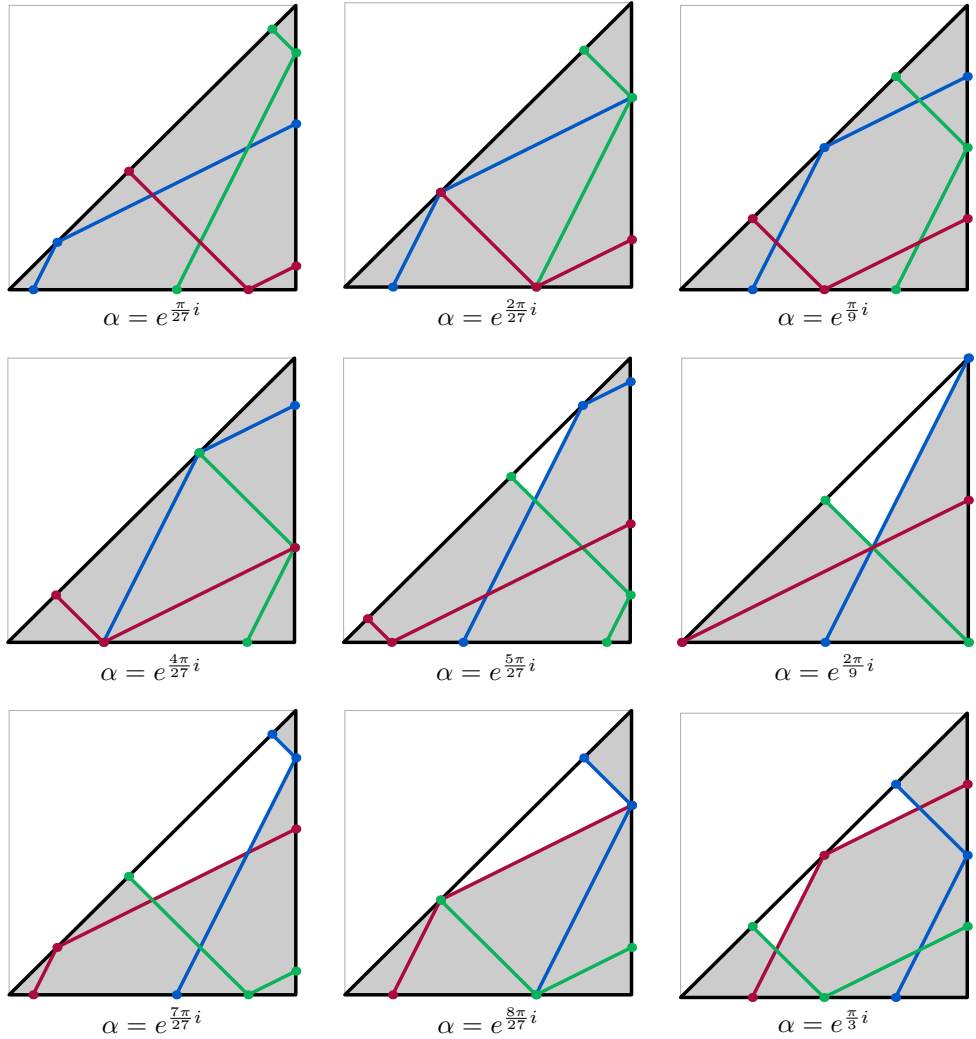
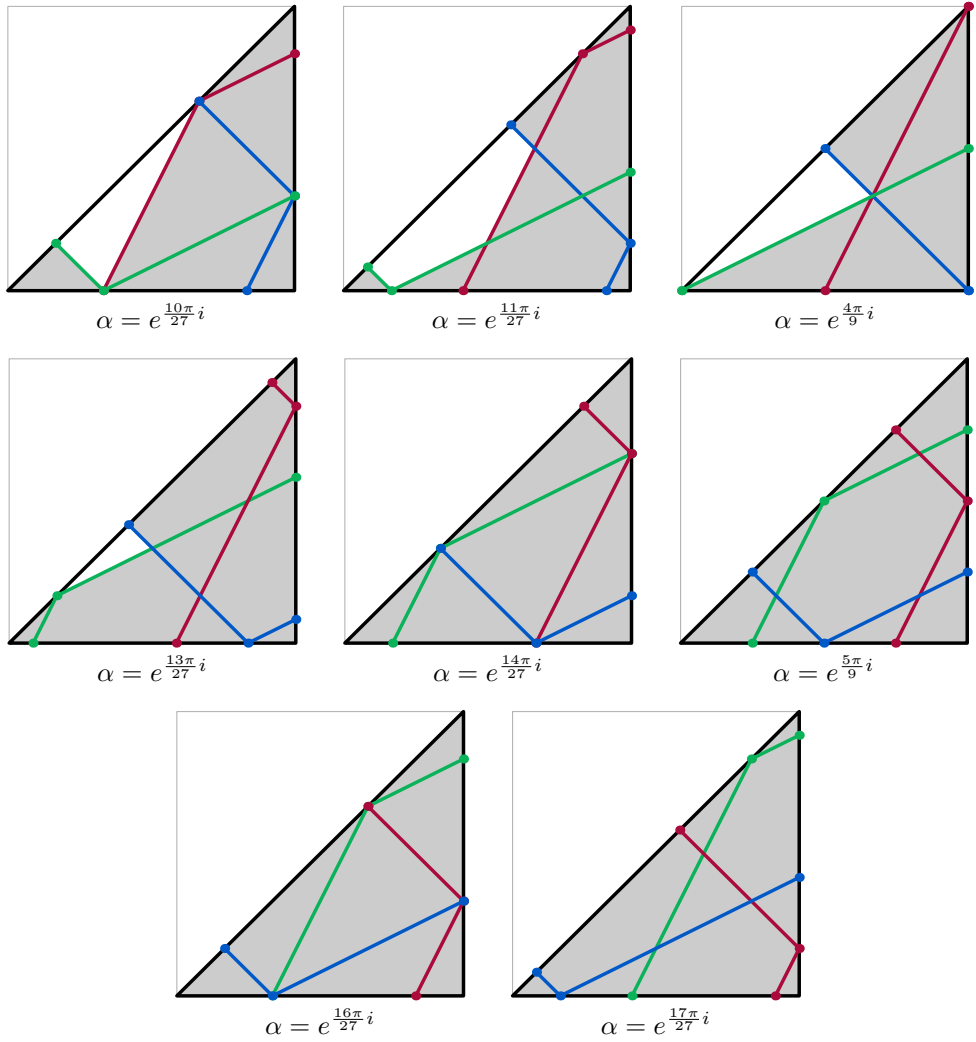
\begin{center}
\includegraphics[scale=1]{pics/fullchambers011.mps}
\hspace{.3cm}
\includegraphics[scale=1]{pics/fullchambers021.mps}
\hspace{.3cm}
\includegraphics[scale=1]{pics/fullchambers031.mps}

\vspace{.3cm}

\includegraphics[scale=1]{pics/fullchambers041.mps}
\hspace{.3cm}
\includegraphics[scale=1]{pics/fullchambers051.mps}
\hspace{.3cm}
\includegraphics[scale=1]{pics/fullchambers061.mps}

\vspace{.3cm}

\includegraphics[scale=1]{pics/fullchambers071.mps}
\hspace{.3cm}
\includegraphics[scale=1]{pics/fullchambers081.mps}
\hspace{.3cm}
\includegraphics[scale=1]{pics/fullchambers091.mps}

\vspace{.3cm}

\includegraphics[scale=1]{pics/fullchambers101.mps}
\hspace{.3cm}
\includegraphics[scale=1]{pics/fullchambers111.mps}
\hspace{.3cm}
\includegraphics[scale=1]{pics/fullchambers121.mps}

\vspace{.3cm}

\includegraphics[scale=1]{pics/fullchambers131.mps}
\hspace{.3cm}
\includegraphics[scale=1]{pics/fullchambers141.mps}
\hspace{.3cm}
\includegraphics[scale=1]{pics/fullchambers151.mps}

\vspace{.3cm}

\includegraphics[scale=1]{pics/fullchambers161.mps}
\hspace{.3cm}
\includegraphics[scale=1]{pics/fullchambers171.mps}
\end{center}

\section{The $\alpha$-length}
We are now able to prove our main theorem. We do so using the
hypothesis that the parameter $\alpha\in\mathbb S^1\setminus\Omega$
is such that $\alpha=e^{ai}$ for some $0<a<\frac{2\pi}3$ which, 
by Corollary~\ref{cor:deltaalpha}, can be assumed without loss of 
generality.

\begin{proof}[Proof of Theorem~\ref{thm:alphalength}]
We start by proving that every isometry that does not admit an $\alpha^{(3)}$-decomposition,
admits an $\alpha^{(4)}$-decomposition. By the results in the
previous sections, we only need to prove this to those elliptic isometries whose
conjugacy classes lie in empty chambers and for $2$-step unipotent isometries.

First, let $F\in\SU(2,1)$ be an elliptic isometry representing a 
$\PU(2,1)$-conjugacy class whose projection under does not lie in $\mathsf E_{\alpha^{(3)}}$.
There exists a point $p\in\PV\setminus\SV$ such that
$R_{\overline\alpha}^{p}F$ is loxodromic. 
In the case where $F$ is special elliptic, this follows 
from~\cite[Corollary~5.10]{spell}.
Suppose $F$ is regular elliptic. Let $c\in\BV$ be the negative $F$-fixed point and let
$L$ be an $F$-stable complex line through $c$.
Denote by $\theta$ the angle in which $F$ rotates points in $L$ around~$c$. 
Given a nonisotropic point $p\in L\cap\BV$, consider the isometry $R:=R_{\overline\alpha}^qF$.
As $R$ acts over $L$ as an isometry of the Poincar\'e disk, the action of $R$ over $L$ can be 
decomposed as the product $r_2r_1$ of reflections $r_1,r_2$ over geodesics $G_1,G_2$ 
through points $c,p$, respectively. If the (dis)tance between $c$ and $p$ is big enough,
the geodesics $G_1,G_2$ are ultraparallel (do not intersect, not even in the absolute $\SV$)
and $R$ acts on $L$ as a hyperbolic isometry of the Poincar\'e disk; thus $R$ is loxodromic.
By Proposition~\ref{prop:parlox3},
$R$ admits an $\alpha^{(3)}$-decomposition, which implies that $F$ admits an
$\alpha^{(4)}$-decomposition.

Now, consider a $2$-step unipotent isometry $U\in\SU(2,1)$. Let $v$ be the isotropic
fixed point of~$U$; then for any nonisotropic point $q\in\mathbb Pv^\perp$,
the isometry $R_{\overline\alpha}^q U$ is ellipto parabolic. It follows
from Proposition~\ref{prop:parlox3} that the isometry $U$ admits an 
$\alpha^{(4)}$-decomposition.  

\smallskip

To prove the second part of the theorem, we use Subsection~\ref{subsec:a3examples}. 
Note that, since the lines $\ell_{\alpha^3}$, $\ell_{\omega\alpha^3}$, and 
$\ell_{\omega^2\alpha^3}$ vary continuously with $\alpha$ and $\alttrace$
is a homeomorphism, it follows that the chambers of $\mathsf E_{\alpha^{(3)}}$ 
vary continuously with~$\alpha$.
Moreover, since the line segments that compose $\mathsf E_{\alpha_1,\alpha_2}$
vary continuously with~$\alpha_1,\alpha_2$ (see Proposition~\ref{prop:a1a2eldecomp}), 
the criteria to determine if a chamber is full/empty
(Propositions~\ref{prop:nondiag} and \ref{prop:diagchamb}) is also continuous, 
i.e., if a chamber of $\mathsf E_{\alpha^{(3)}}$ is full/empty for a given $\alpha$, it 
continues to be full/empty for 
parameters sufficiently close to $\alpha$. 
It follows that if a chamber if full/empty, it continues to be full/empty until 
it disappears. So, we need to determine the transition parameters (the parameters
where chambers appear or disappear).

The transition parameters are those in which the lines $\ell_{\alpha^3},
\ell_{\omega\alpha^3},\ell_{\omega^2\alpha^3}$ either pairwise intersect
at a point where one of then is tangent to $\partial\Delta$ or
all intersect at the same point (in this case such point must be $0\in\mathbb C$ and
the lines are tangent to a vertex of $\partial\Delta$). In the first case,
by Proposition~\ref{prop:traceeigen}, we must have $\alpha^3=\omega\alpha^{-6}$ 
which implies that $a=0\,(\mathrm{mod}\,\frac{2\pi}{27})$. In the second case, we must have
$a=0\,(\mathrm{mod}\,\frac{2\pi}9)$. Thus, the transition angles are those satisfying
$a\neq 0\,(\mathrm{mod}\,\frac{2\pi}{27})$. In other words,
if we are looking at the chambers of $\mathsf E_{\alpha^{(3)}}$
while continuously increasing the value of $a$ (remember that $\alpha=e^{ai}$),
chambers appear or disappear while passing through a parameter such that
$a=0\,(\mathrm{mod}\,\frac{2\pi}{27})$.

Therefore, the second part follows from the cases we obtained in 
Subsection~\ref{subsec:a3examples}; the last part is just 
Proposition~\ref{prop:2stepdecomp}.
\end{proof}

As the isometries that contribute to the $\alpha$-length, $\alpha=e^{ai}$, 
not being~$3$ when $0<a<\frac{4\pi}{27}$ are only, possibly, the $2$-step 
unipotent ones (see Remark~\ref{rmk:2stepdecomp}), which are not semisimple, we 
have the following result (see Subsection~\ref{subsec:puclasses} for definitions). 

\begin{cor}
The\/ $\alpha$-length of the space of semisimple conjugacy classes in\/ $\mathcal G$ is

$\bullet$ $3$, if\/ $0<a<\frac{4\pi}{27}$ or\/ $\frac{14\pi}{27}<a<\frac{2\pi}{3}$;

$\bullet$ $4$, if\/ $\frac{4\pi}{27}\leq a\leq\frac{14\pi}{27}$.
\end{cor}

\bibliographystyle{plain}
\bibliography{references-alpha-final}

\noindent
{\sc Felipe A.~Franco}

\noindent
{\sc Departamento de Matem\'atica, ICMC, Universidade de S\~ao Paulo, Brasil}

\noindent
\url{felipefranco@usp.br}, \url{f.franco.math@gmail.com}

\end{document}